\documentclass[reqno, 11pt]{amsart}
\usepackage{amsmath,mathtools}
 \usepackage{amssymb}
\usepackage{amsthm}
\usepackage{amsmath}
\usepackage{times}
\usepackage{latexsym}
\usepackage[mathscr]{eucal}

\numberwithin{equation}{section}
 
  \newtheorem{theorem}{Theorem}[section]
  \newtheorem{proposition}[theorem]{Proposition}
  
  \newtheorem{corollary}[theorem]{Corollary}

  \newtheorem{definition}[theorem]{Definition}
  \newtheorem{example}[theorem]{Example}

\title[Higher order mean curvatures of SAC half-lightlike submanifolds]{Higher order mean curvatures of SAC half-lightlike submanifolds of indefinite almost contact manifolds}

\author[Fortun\'{e} Massamba, Samuel Ssekajja ]{Fortun\'{e} Massamba*,   Samuel Ssekajja**}

\newcommand{\acr}{\newline\indent}

\address{\llap{*\,} School of Mathematics, Statistics and Computer Science\acr
 University of KwaZulu-Natal\acr
 Private Bag X01, Scottsville 3209\acr
South Africa}
\email{massfort@yahoo.fr, Massamba@ukzn.ac.za} 
\thanks{}

\address{\llap{**\,} School of Mathematics, Statistics and Computer Science\acr
 University of KwaZulu-Natal\acr
 Private Bag X01, Scottsville 3209\acr
South Africa}
\email{ssekajja.samuel.buwaga@aims-senegal.org} 
\thanks{}



\begin{document}
  
\begin{abstract}
We introduce higher order mean curvatures of screen almost conformal (SAC) half-lightlike submanifolds of indefinite almost contact  manifolds, admitting a semi-symmetric non-metric connection,  and use them to generalize some known results of \cite{dusa1}. Also, we derive a new set of integration formulae via the divergence of some special vector fields tangent to this submanifold. Several examples are also included to illustrate the main concepts.
\end{abstract}

\maketitle

\section{Introduction} 
Null (or lightlike) subspaces exist naturally in semi-Riemannian spaces and they play a central role in general relativity. More precisely,  in the study of black holes (small volumes of spacetime with infinite density). In fact, they are subspaces whose induced metrics are singular (or simply with vanishing determinants). Differential geometry of these subspaces was introduced by Duggal and Bejancu in their book \cite{db}, which was later updated by Dugal and Sahin to \cite{ds2}. Their approach was later adopted by many other researchers, including but not limited to; \cite{ati}, \cite{dusa1}, \cite{gup1}, \cite{ma1}, \cite{ma2}, \cite{ma22} and  \cite{cohen}. From the above pieces of work, we can see the theory of lightlike geometry rests on a number of operators, including, shape, Ricci, etc., together with functions constructed from them, like mean curvature, scalar curvature, etc. However, the most important of such functions are the ones derived from algebraic invariants of their respective operators. For instance, trace, determinant, and in a more general sense the $r^{th}$ symmetric functions, $\sigma_{r}$. These functions play a central role in studying higher order mean curvatures in differential geometry of both Riemannian and semi-Riemannian manifolds. In fact, for any given point in a manifold, the $r^{th}$ symmetric function $\sigma_{r}$ coincides with the $r^{th}$ mean curvature $S_{r}$. A great deal of work has been done for $r=1$ (see \cite{db}, \cite{ds2}, \cite{dusa1}, \cite{ma1} and references therein). But the case $r>1$ is strictly non-linear and complicated. The most efficient way of studying this case is the use of Newton transformations, $T_{r}$, of a given operator $A$ (or a system of operators) which, in some way, linearizes $S_{r}$. That is to say  $(-1)^{r-1}rS_{r}=\mathrm{tr}(A\circ T_{r-1})$.

Let $M^{m+1}$ be a half lightlike submanifold of an indefinite contact manifold $\overline{M}^{m+3}$ admitting a semi-symmetric non-metric connection. Then, $M$ carries three shape operators $A_{E}^{*}$, $A_{N}$ and $A_{W}$, where $E$, $N$ and $W$ are respectively vector fields in its radical distribution, lightlike transversal bundle and screen transversal bundle. When the structure vector field $\xi$ is tangent to $M$ but not necessarily in its screen distribution, then $A_{E}^{*}$ is a self-adjoint operator on $TM$ while $A_{N}$ and $A_{W}$ are generally not self-adjoint. If we suppose that $M$ is a screen almost conformal (SAC) \cite{ma3} half-lightlike submanifold , then the operator $A_{N}$ becomes self-adjoint on $TM$ and therefore diagonalizable on $TM$ and hence we can investigate its higher order mean curvatures. With such mean curvatures, one can also investigate integration  geometry on such submanifolds. Integration geometry is fundamentally important as it provide obstructions to the existence of foliations whose leaves enjoy some special geometric properties-totally geodesic (or totally umbilical), minimal, constant mean curvature and many more. Also, it provides a way of minimizing volume (of submanifolds)  as well as energy defined from smooth vector fields on manifolds (see \cite{woj} and references therein).  

In  this paper, we consider SAC half lightlike submanifold $M$ of an indefinite contact manifold $\overline{M}$, admitting a semi-symmetric non-metric connection. We derive equations relating the $r^{th}$ mean curvatures and Newton transformations of $A_{N}$ and $A_{E}^{*}$. We generalize some known results for $r=1$ and also, derive new integration formulas by computing the divergence of some vector fields on in the tangent bundle of $M$. The rest of the paper is arranged as follows. Section \ref{prel} outlines the basic preliminary concepts needed in other parts of the paper. Section \ref{New1} introduces Newton transformations of $A_{E}^{*}$. In Section \ref{dist} we show that the $r^{th}$ mean curvatures and Newton transformations of $A_{N}$ and $A_{E}^{*}$ are in partial variation (see Proposition \ref{PROP1} and Theorem \ref{PROP3}). Also, we derive generalized differential equations for $r^{th}$ mean curvatures (Theorem \ref{Theo1}). In Section \ref{int} we present  special integration formulas by computing the divergence of some vector fields (see Theorem \ref{The4} and its corollaries).

\section{Preliminaries}\label{prel}

Let $\overline{M}$ be a $(2n + 1)$-dimensional manifold endowed with an almost contact structure $(\overline{\phi}, \xi, \eta)$, i.e. $\overline{\phi}$ is a tensor field of type $(1, 1)$, $\xi$ is a vector field, and $\eta$ is a 1-form satisfying
\begin{equation}\label{equa1}
\overline{\phi}^{2} = -\mathbb{I} + \eta
\otimes\xi,\;\;\eta(\xi)= 1 ,\;\;\eta\circ\overline{\phi} =
0 \;\;\mbox{and}\;\;\overline{\phi}(\xi) = 0.
\end{equation}
Then $(\overline{\phi}, \xi, \eta,\,\overline{g})$ is called an indefinite almost contact metric structure on $\overline{M}$ if $(\overline{\phi}, \xi, \eta)$ is an almost contact structure on $\overline{M}$ and $\overline{g}$ is a semi-Riemannian metric on $\overline{M}$ such that \cite{bl2}, for any vector field $\overline{X}$, $\overline{Y}$ on  $\overline{M}$,
\begin{equation}\label{equa2}
 \overline{g}(\overline{\phi}\,\overline{X}, \overline{\phi}\,\overline{Y}) = \overline{g}(\overline{X}, \overline{Y}) -  \eta(\overline{X})\,\eta(\overline{Y}),
\end{equation}
It follows that, for any vector $\overline{X}$ on $\overline{M}$,  
$
  \eta(\overline{X}) =  \overline{g}(\xi,\overline{X}).
$
We denote by $\Gamma(\Xi)$ the set of smooth sections of the vector bundle $\Xi$.
 
 A connection $\overline{\nabla}$ on $\overline{M}$ is called a \textit{semi-symmetric non-metric connection} \cite{Jin,cohen} if $\overline{\nabla}$ and its corresponding torsion tensor $\overline{T}$ satisfy the following equations;
\begin{align}\label{N1}
 (\overline{\nabla}_{\overline{X}}\overline{g})(\overline{Y},\overline{Z}) & =-\eta(\overline{Y})\overline{g}(\overline{X},\overline{Z})-\eta(\overline{Z})\overline{g}(\overline{X},\overline{Y}), \\ \label{N2}
\mbox{and}\;\; \overline{T}(\overline{X},\overline{Y})& =\eta(\overline{Y})\overline{X}-\eta(\overline{X})\overline{Y}, 
\end{align}
for all $\overline{X}$, $\overline{Y}$ and $\overline{Z}$ vector fields on $\overline{M}$.  

Let $(\overline{M},\overline{g})$ be an $(m + n)$-dimensional semi-Riemannian manifold of constant index $\nu$, $1\le \nu<m+n$ and $M$ be a submanifold of $\overline{M}$ of codimension $n$. We assume that both $m$ and $n$ are $\ge 1$. At a point $p\in M$, we define the orthogonal complement $T_{p} M^{\perp}$ of the tangent space $T_{p} M$ by
 $ 
 T_{p} M^{\perp} = \{X\in\Gamma(T_{p} M): \overline{g}(X, Y)=0,\,\forall \, Y\in\Gamma( T_{p} M)\}.
 $ 
Take $\mathrm{Rad} \, T_{p} M = \mathrm{Rad}\, T_{p} M^{\perp} = T_{p} M \cap T_{p} M^{\perp}$.

The submanifold $M$ of $\overline{M}$ is said to be \textit{$r$-lightlike submanifold}, if the mapping 
 $ 
 \mathrm{Rad} \, T M: p\in M \longrightarrow\mathrm{Rad}\, T_{p} M, 
 $ 
 defines a smooth distribution on $M$ of rank $r > 0$. We call $\mathrm{Rad}\,T M$ the radical distribution on $M$. 
 
 We say that $M$ is a \textit{half-lightlike submanifold} of $\overline{M}$ \cite{ds2} if $r=1$, $n=2$ and there exist $E,W\in\Gamma(T_{p}M^{\perp})$ such that 
 \begin{equation}\label{N3}
  \overline{g}(E,V)=0,\;\;\;\overline{g}(W,W)\neq 0,\;\; \forall\,V\in \Gamma(T_{p}M^{\perp}).
 \end{equation}
From (\ref{N3}), we observe that $E\in \mathrm{Rad}\, T_{p}M$ and therefore, 
\begin{equation}\label{N4}
  \overline{g}(E,X)=\overline{g}(E,V)= 0,\;\;\forall\, X\in \Gamma(T_{p}M),\;\;V\in\Gamma(T_{p}M^{\perp}).
\end{equation}
Thus, $\mathrm{Rad}\, TM$ is locally (or globally) spanned by $E$.  

Let $S(T M)$ be a screen distribution which is a semi-Riemannian complementary distribution of $\mathrm{Rad}\,T M$ in $T M$, that is,
 \begin{equation}\label{N5}
  T M = \mathrm{Rad}\,T M \perp S(T M).
 \end{equation} 
Choose a screen transversal bundle $S(TM^\perp)$, which is semi-Riemannian and complementary to $\mathrm{Rad}\, TM$ in $TM^\perp$. Since, for any null section $E$ of  $\mathrm{Rad}\,TM$, there exists a unique null section $N$ of the orthogonal complement of $S(T M^\perp)$ in $S(T M )^\perp$  such that $g(E,N) = 1$, it follows that there exists a lightlike transversal vector bundle $l\mathrm{tr}(TM)$ locally spanned by $N$ \cite{db}. Let $W\in\Gamma(S(TM^{\perp}))$ be a unit vector field, then $\overline{g}(N,N) = \overline{g}(N, Z)=\overline{g}(N, W) = 0$, for any $Z\in \Gamma(S(TM))$.
 
Let $\mathrm{tr}(TM)$ be complementary (but not orthogonal) vector bundle to $TM$ in $T\overline{M}$. Then, 
\begin{align}
 \mathrm{tr}(TM)& =l\mathrm{tr}(TM)\perp S(TM^\perp),\label{N8}\\
  T\overline{M} & =  S(TM)\perp S(TM^\perp)\perp\{\mathrm{Rad}\, TM\oplus l\mathrm{tr}(TM)\}\label{N9} .
\end{align}
Note that the distribution $S(TM)$ is not unique, and is canonically isomorphic to the factor vector bundle $TM/ \mathrm{Rad}\, TM$  \cite{db}. 

Let $P$ be the projection of $TM$ on to $S(TM)$. Throughout this paper,  we shall suppose that $\xi$ is a unit space-lightlike vector field. Moreover,  from (\ref{N9}) $\xi$ is decomposed  as follows
\begin{equation}\label{N7}
 \xi=\xi_{S}+aE+bN+eW,
\end{equation}
where $\xi_{S}$ denotes the projection of the tangential part of $\xi$ on to $S(TM)$ and $a=\eta(N)$, $b=\eta(E)$ and $e=\epsilon\eta(W)$, with $\epsilon=\pm 1$, are smooth functions on $\overline{M}$.
The Gauss and Weingarten formulas are given by
\begin{align}
 \overline{\nabla}_{X}Y&=\nabla_{X}Y+h(X,Y),\;\;\forall\,X,Y\in\Gamma(TM)\label{N10}\\
 \overline{\nabla}_{X}V=-&A_{V}X+\nabla_{X}^tV,\;\;\forall\,X\in\Gamma(TM),\;\;V\in\Gamma(\mathrm{tr}\,(T M)).\label{N11}
\end{align}
Notice that $\{\nabla_{X}Y,A_{V}X\}$ and $\{h(X,Y),\nabla_{X}^tV\}$ belongs to $\Gamma(TM)$ and $\Gamma(\mathrm{tr}(T M))$ respectively. Further,  $\nabla$ and  $\nabla^{t}$ are linear connections on $M$ and $\mathrm{tr}\,T M$, respectively. The second fundamental form $h$ is a symmetric $\mathscr{F}(M)$-bilinear form on $\Gamma(T M)$ with values in $\Gamma(\mathrm{tr}(T M))$ and the shape operator $A_{V}$ is a linear endomorphism of $\Gamma(T M )$. Then, for all  $X,Y\in\Gamma(TM)$, (\ref{N10}) and (\ref{N11}) gives
\begin{align}
 &\overline{\nabla}_{X}Y=\nabla_{X}Y+B(X,Y)N+D(X,Y)W,\label{N12}\\
 &\overline{\nabla}_{X}N=-A_{N}X+\tau(X)N+\rho(X)W,\label{N13}\\
 &\overline{\nabla}_{X}W=-A_{W}X+\phi(X)N,\label{N14}\\
  &\nabla_{X}PY = \nabla^{*}_{X}PY + C(X,PY)E,\label{N15}\\
  &\nabla_{X}E =-A^{ * }_{E}X -\delta(X) E,\label{N16}
 \end{align}
 for all $E\in\Gamma(\mathrm{Rad}\,T M)$, $N\in\Gamma(l\mathrm{tr}(T M))$ and $W\in\Gamma(S(TM^{\perp}))$, where 
 \begin{equation*}
 h(X,Y)=B(X,Y)N+D(X,Y)W,
 \end{equation*}
 $C$  is the local second fundamental form on $S(T M)$, $\{A_{N}, A_{W}\}$ and $A^{*}_{E}$ are the shape operators on $TM$ and $S(TM)$ respectively, and $\tau$, $\rho$, $\phi$ and $\delta$ are differential 1-forms on $TM$. Notice that $\nabla^{*}$ is a metric connection on $S(TM)$ while $\nabla$ is generally not a metric connection. In fact, using (\ref{N1}) and (\ref{N12}), we deduce 
 \begin{align}\label{N17}
  (\nabla_{X}g)(Y,Z)&= B(X,Y)\lambda(Z) + B(X,Z)\lambda(Y)\nonumber\\
                    &- \eta(Y )g(X,Z)-\eta(Z)g(X,Y),
 \end{align} 
 for all $X,Y,Z\in\Gamma(TM)$, where $\lambda$ is a 1-form on $TM$ given  $\lambda(\cdot)= \overline{g}(\cdot, N)$.  
It is well known \cite{db, ds2} that $B$ and $D$ are independent of the choice of $S(TM)$ and they satisfy
\begin{equation}\label{N18}
 B(X, E) = 0,\;\;\;D(X, E) = −\phi(X),\;\;\forall\,  X \in\Gamma(TM).
\end{equation}
The three local second fundamental forms $B$, $D$ and $C$ are related to their shape operators by the following equations
 \begin{align}
  &g(A^{*}_{E}X,Y)=B(X,Y)-bg(X,Y),\;\;\;\;\; \overline{g}(A^{*}_{E}X,N)=0,\label{N19}\\
  &g(A_{W}X,Y) = D(X,Y)- eg(X,Y)+ \phi(X)\lambda(Y),\label{N20}\\
  &g(A_{N}X,PY) = C(X,PY )-ag(X,PY )-\lambda(X)\eta(PY),\label{N21}\\
  &\overline{g}(A_{N} X,N) = -a\lambda(X),\quad \overline{g}(A_{W} X,N) = \rho(X)-e\lambda(X),\label{N361}\\
  &\delta(X) = \tau(X)-b\lambda(X),\quad \forall\, X,Y\in\Gamma(TM).\label{N22}
 \end{align}
 From the first equations of (\ref{N18}) and (\ref{N19}) we deduce that $A^{*}_{E}$ is $S(TM)$-valued, self-adjoint and satisfies 
$
  A^{*}_{E}E=0.
$
Let $\overline{R}$ and $R$ denote the curvature tensors of $\overline{M}$ and $M$ respectively. Then,  using the Gauss-Weingarten equtions for $M$, we derive
\begin{align}\label{N37}
 &\overline{R}(X,Y)Z =R(X,Y)Z+ B(X,Z)A_{N}Y - B(Y,Z)A_{N}X\nonumber\\
                   &+ D(X,Z)A_{W}Y - D(Y,Z)A_{W}X +\{(\nabla_{X}B)(Y,Z) - (\nabla_{Y}B)(X,Z)\nonumber\\
                   &+ \tau(X)B(Y,Z) -\tau (Y)B(X,Z) + \phi(X)D(Y,Z) - \phi(Y)D(X, Z)\}N\nonumber\\
                   &+\{(\nabla_{X}D)(Y,Z) - (\nabla_{Y}D)(X,Z) + \rho(X)B(Y,Z) - \rho(Y)B(X,Z)\}W, 
     \end{align}              
  \begin{align} \label{N23}
 \overline{R}(X,Y)N &=-\nabla_{X}(A_{N}Y)+\nabla_{Y}(A_{N}X)+A_{N}[X,Y]\nonumber\\
                    &+\tau(X)A_{N}Y-\tau(Y)A_{N}X+\rho(X)A_{W}Y\nonumber\\
                    &-\rho(Y)A_{W}X+\{B(X,A_{N}X)-B(X,A_{N}Y)\nonumber\\
                    &+2d\tau(X,Y)+\phi(X)\rho(Y)-\phi(Y)\rho(X)\}N\nonumber\\
                    &+\{D(X,A_{N}X)-D(X,A_{N}Y)+2d\rho(X,Y)\nonumber\\
                    &+\rho(X)\tau(Y)-\rho(Y)\tau(X)\}W, \;\;\;\forall\,X,Y\in\Gamma(TM).
                    \end{align}
A half lightlike submanifold $M$ of an indefinite almost contact manifold $\overline{M}$, with $\xi\in\Gamma(TM)$, is called \textit{screen almost conformal (SAC)} \cite{ma3} if  the shape operators $A_{N}$ and $A^{*}_{E}$ of $M$ and $S(TM)$, respectively, are linked to each other by 
\begin{equation}\label{conf}
 A_{N}=\varphi A^{*}_{E}+\lambda\otimes\xi,
\end{equation}
or equivalently
\begin{equation}\label{N25}
 C(X,PY)=\varphi B(X,Y)+\lambda(X)\eta(PY), \;\;\forall\, X,Y\in\Gamma(TM),
\end{equation}
where $\varphi$ is a non vanishing smooth function on a coordinate neighborhood $\mathcal{U}$ of $M$. Furthermore, $M$ is \textit{screen almost homothetic} if $\varphi$ is a non-vanishing constant function.

When $\overline{\nabla}$ is a metric connection, it is easy to show that $\overline{g}(A_{N}X,N)=0$ for any $X\in\Gamma(TM)$, for any half-lightlike submanifolds. Hence $A_{N}$ is screen-valued operator and thus, the screen almost conformallity condition (\ref{conf}) makes sense only if $\xi\in\Gamma(S(TM))$. 
 
In the following example, we  consider the connection in the ambient space to be Levi-Civita and construct a SAC half-lightlike submanifold of an indefinite Kenmotsu manifold.

 Hereunder, we consider a manifold $\overline{M}=(\mathbb{R}_{q}^{2m+1}, \overline{\phi}_{0} ,\xi, \eta, \overline{g})$  with its usual Kenmotsu structure given by;
\begin{align*}
   \eta &= dz,\quad\xi=\partial z,\\
   \overline{g} &=\eta\otimes\eta-e^{-2z}\sum_{i=1}^{\frac{q}{2}}(dx^i\otimes dx^i+dy^i\otimes dy^i)\\
   &+e^{-2z}\sum_{i=q+1}^{m}(dx^i\otimes dx^i+dy^i\otimes dy^i),\\
   \overline{\phi}_0& (\sum_{i=1}^m(X_i\partial x^i+Y_i\partial y^i)+Z\partial z )=\sum_{i=1}^m(Y_i\partial x^i-X_i\partial y^i),
\end{align*}
where $(x^{i} , y^{i} , z)$ are Cartesian coordinates and $\partial t_{k} = \frac{\partial}{\partial t^{k}}$, for $t\in\mathbb{R}^{2m+1}$.

\begin{example}\label{SAC}
 {\rm
 Let $\overline{M}=(\mathbb{R}_{2}^{9}, \overline{g})$ be a semi-Euclidean space,  where $\overline{g}$ is of signature $(-,+,+,+, -,+,+,+,+)$ with respect to the canonical basis
$$
 (\partial x_{1},\partial x_{2},\partial x_{3},\partial x_{4},\ \partial y_{1},\partial y_{2},\partial y_{3},\partial y_{4},\partial z).
$$
Let $\overline{\nabla}$ be the Levi-Civita connection with respect to the semi-Riemannian metric $\overline{g}$ and consider the vector fields $u_{1},\cdots,u_{9}$, where for all $1 \le i \le 8$ we have  
\begin{align*}
u_{i}&=e^{z}\sum_{\alpha'=1}^{4}f_{i\alpha'}(x_{1}, \cdots, x_{4},y_{1}, \cdots, y_{4},z)\partial x_{\alpha'}\\
&+e^{z}\sum_{\beta'=\alpha' +1}^{8}f_{i\beta'}(x_{1}, \cdots, x_{4},y_{1}, \cdots, y_{4},z)\partial y_{\beta'},\;\;\; \det(f_{ij}) \neq 0, \\
u_{9} & = -\xi,
\end{align*}
 where functions $f_{i\alpha'}$ and $f_{i\beta'}$ are defined such that the action of the connection $\overline{\nabla}$ on the basis $\{u_{1},\cdots,u_{9}\}$ gives $\overline{\nabla}_{u_{i}}u_{i}=\xi$ for all $i=1,5$, $\overline{\nabla}_{u_{j}}u_{j}=-\xi$ for all $j=2,3,4,6,7,8$, $\overline{\nabla}_{u_{3}}u_{1}=e^{z}u_{5}$, $\overline{\nabla}_{u_{3}}u_{5}=-e^{z}u_{1}$, $\overline{\nabla}_{u_{7}}u_{1}=e^{z}u_{5}$, $\overline{\nabla}_{u_{7}}u_{5}=-e^{z}u_{1}$, $\overline{\nabla}_{u_{1}}u_{3}=e^{z}u_{5}$, $\overline{\nabla}_{u_{1}}u_{7}=e^{z}u_{5}$, $\overline{\nabla}_{u_{5}}u_{3}=-e^{z}u_{1}$, $\overline{\nabla}_{u_{5}}u_{7}=e^{z}u_{1}$ and the rest of the connections $\overline{\nabla}_{u_{i}}u_{j}=0$ for all $i\neq j$, where $,i,j=1,\cdots,8$. Furthermore, the non-vanishing brackets are $[u_{i},u_{9}]=u_{i}$, for all $i=1,\cdots,8$. Using Koszul's formula, we have $\overline{\nabla}_{u_{i}}u_{9}=u_{i}$ for all $i=1,\cdots,8$ and $\overline{\nabla}_{u_{9}}u_{9}=0$. From these constructions, $(\overline{\phi}_{0}, u_{9}, \eta, \overline{g})$ defines n almost contact structure on $\mathbb{R}^{9}_{2}$. Therefore, $(\mathbb{R}^{9}_{2},\overline{\phi}_{0}, u_{9}, \eta, \overline{g})$ is an indefinite Kenmotsu manifold. 
 
 Next, let us consider a submanifold $M$ of $\mathbb{R}^{9}_{2}$ above which is given by the following equation $x_{1}=\sqrt{2}(x_{2}+y_{2})$. By straightforward calculations, one can easily show that the vectors $E=\frac{1}{\sqrt{2}}(u_{6}+u_{2})-u_{1}$, $Z_{1}=u_{3}$, $Z_{2}=u_{7}$, $Z_{3}=u_{6}-u_{2}$, $Z_{4}=u_{5}$, $Z_{5}=u_{8}$ and $Z_{6}=u_{9}=-\xi$ form a local frame of $TM$. Clearly, $\mathrm{Rad}\, TM=\mathrm{span}\{E\}$ and $S(TM)=\mathrm{span}\{Z_{1},\cdots,Z_{6}\}$. Also, the lightlike transversal bundle $l\mathrm{tr}(TM)$ and co-screen $S(TM^{\perp})$ are respectively spanned by $N$ and $W$, where $N=\frac{1}{2\sqrt{2}}(u_{6}+u_{2})+\frac{1}{2}u_{1}$ and $W=U_{4}$.  Thus, $M$ is a half-lightlike submanifold of $(\mathbb{R}^{9}_{2},\overline{\phi}_{0}, u_{9}, \eta, \overline{g})$. By straightforward calculations, we have $\overline{\nabla}_{Z_{1}}N=\frac{1}{2}e^{z}u_{5}$, $\overline{\nabla}_{Z_{2}}N=\frac{1}{2}e^{z}u_{5}$, $\overline{\nabla}_{Z_{3}}N=\overline{\nabla}_{Z_{4}}N=\overline{\nabla}_{Z_{5}}N=\overline{\nabla}_{Z_{6}}N=0$, $\overline{\nabla}_{E}N=-u_{9}=\xi$. Furthermore  $\nabla_{Z_{1}}E=-\frac{1}{2}e^{z}u_{5}$, $\nabla_{Z_{2}}E=-\frac{1}{2}e^{z}u_{5}$, $\nabla_{Z_{3}}E=\nabla_{Z_{4}}E=\nabla_{Z_{5}}E=\nabla_{Z_{6}}E=\nabla_{E}E=0$. From these connections, the 1-forms $\tau$ and $\rho$ vanish on $TM$. Therefore, from (\ref{N14}) and (\ref{N16}) we deduce that $A_{N}Z_{1}=-\frac{1}{2}e^{z}u_{5}$, $A_{N}Z_{2}=-\frac{1}{2}e^{z}u_{5}$, $A_{N}Z_{3}=A_{N}Z_{4}=A_{N}Z_{5}=A_{N}Z_{6}=0$, $A_{N}E=u_{9}=-\xi$. Also, $A_{E}^{*}Z_{1}=e^{z}u_{5}$, $A_{E}^{*}Z_{2}=e^{z}u_{5}$, $A_{E}^{*}Z_{3}=A_{E}^{*}Z_{4}=A_{E}^{*}Z_{5}=A_{E}^{*}Z_{6}=A_{E}^{*}E=0$. From these relations, we deduce that $A_{N}X=\varphi A_{E}^{*}X+\lambda(X)\xi$ for all $X\in\Gamma(TM)$, with $\varphi=-\frac{1}{2}$. Hence, $M$ is SAC half-lightlike submanifold of $(\mathbb{R}^{9}_{2},\overline{\phi}_{0}, u_{9}, \eta, \overline{g})$.

 }
 \end{example}
 When $\overline{\nabla}$ is semi-symmetric non-metric connection, one can easily verify that  $\overline{g}(A_{N}X,N)=-\eta(N)\lambda(X)$, for any $X\in\Gamma(TM)$. See (\ref{N361}) for details. This shows that $A_{N}$ is generally not a screen-valued operator. Thus, the screen almost conformality condition  (\ref{conf}) allows a $\xi\in\Gamma(TM)$ but not necessarily in $S(TM)$, given by $\xi=\xi_{S}+aE$. 
Since on any SAC half-lightlike submanifold we have $\xi\in\Gamma(TM)$, then it is easy to see from 
(\ref{N7}) that $b=0$ and $e=0$. Hence, using  (\ref{N21}) and (\ref{N25}) we deduce that $M$ is SAC half-lightlike submanifold of $\overline{M}$ admitting a semi-symmetric non-metric connection if and only if 
\begin{equation}\label{N26}
 A_{N}X=\varphi A^{*}_{E}X-aX,\;\;\;\forall\,X\in\Gamma(TM).
\end{equation}
Furthermore, if $M$ is SAC half-lightlike submanifold of an almost contact manifold $\overline{M}$ admitting a semi-symmetric non-metric connection then from (\ref{N26}), the operator $A_{N}$ is self-adjoint on $TM$ and thus diagonalizable on $TM$. Also, from (\ref{N26}) we have that any SAC half-lightlike submanifold of an almost contact manifold admitting a semi-symmetric non-metric connection, with the structure vector field $\xi\in\Gamma(S(TM))$ is screen conformal.

\section{Newton transformations of $A_{E}^{*}$} \label{New1}

Let $\overline{M}$ be an $(n+3)$-dimensional almost contact metric manifold admitting a semi-symmetric non-metric connection and $(M,g,S(TM))$ be a codimension two SAC half-lightlike submanifold of $\overline{M}$. Since $A^{*}_{E}$ is a self-adjoint operator,  from (\ref{N19}) and (\ref{N26}) we can see that $A_{N}$ is a  self-adjoint linear operator on $TM$. Thus, $A_{E}^{*}$ and $A_{N}$ are diagonalizable. Hence, $A_{E}^{*}$ has  $(n+1)$ real eigenvalues  $\kappa_{0}^{*}=0, \kappa_{1}^{*},\cdots, \kappa_{n}^{*}$ (the principal curvatures) corresponding to a set of quasi-orthonormal frame field of eigenvector fields $\{Z_{0}=E,Z_{1},\cdots,Z_{n}\}$. By the SAC condition (\ref{N26}) it is easy to see that $-a,(\varphi \kappa_{1}^{*}-a),\cdots, (\varphi \kappa_{n}^{*}-a)$  are eigenvalues of $A_{N}$. Moreover the matrix of $A_{N}$ has the form
\begin{equation*}
 A_{N} = \mathrm{diag}(-a, \varphi \kappa_{1}^{*}-a, \cdots, \varphi \kappa_{n}^{*}-a ).
\end{equation*}
Associated to the shape  operator $A_{E}^{*}$ are $(n+1)$ algebraic invariants
\begin{equation*}
 S_{r}^{*}=\sigma_{r}(\kappa_{0}^{*}, \kappa_{1}^{*},\cdots, \kappa_{n}^{*}),
\end{equation*}
where $\sigma_{r}: M^{n+1}\rightarrow\mathbb{R}$, for $r=0,1,\cdots, n+1$, are symmetric functions given by
\begin{equation}\label{N71}
 \sigma_{r}(\kappa_{0}^{*}, \kappa_{1}^{*},\cdots, \kappa_{n}^{*})=\sum_{0\leq i_{1}< \cdots< i_{r}\leq n} \kappa_{i_{1}}^{*}\cdots \kappa_{i_{r}}^{*}.
\end{equation}
Let denote by $\mathbb{I}$ the identity map in $\Gamma(TM)$. Then, the  characteristic polynomial of $A_{E}^{*}$ is given by 
\begin{equation*}
 Q(t)=\det(A_{E}^{*}-t\mathbb{I})=\sum_{\alpha=0}^{n+1}(-1)^{\alpha}S_{r}^{*}t^{n+1-\alpha}.
\end{equation*}
The normalized $r$-th mean curvature $H_{r}^{*}$ of $M$ is defined by
\begin{equation*}
 H_{r}^{*}=\dbinom{n+1}{r}^{-1}S_{r}^{*}\;\;\;\mbox{and}\;\;\;H_{0}^{*}=1. \;\;\mbox{(a constant function 1)}.
\end{equation*}
In particular, when $r=1$, then $H_{1}^{*}=\frac{1}{n+1} \mathrm{tr}(A_{E}^{*})$ which is called the \textit{mean curvature} of the half-lightlike submanifold $M$. On the other hand, $H_{2}^{*}$ relates directly with the (intrinsic) scalar curvature of $M$. Often times, $H_{r}^{*}$ instead of $S_{r}^{*}$, is called the $r$-th mean curvature \cite{ati, krz}. Moreover, the functions $S_{r}^{*}$ ($H_{r}^{*}$ respectively) are smooth on the whole $M$ and, for any point $p\in M$, $S_{r}^{*} $ coincides with the $r$-th mean curvature at $p$. Throughout this paper, we shall use $S_{r}^{*}$ instead of $H_{r}^{*}$.

The Newton transformations $T_{r}^{*}:\Gamma(TM)\rightarrow \Gamma(TM)$, for $r=0,1,\cdots,n+1$, of a SAC half-lightlike submanifold $M$ of an $(n+3)$-dimensional almost contact metric manifold $\overline{M}$ with respect to $A_{N}$ are given by 
\begin{equation}\label{N78}
 T_{r}^{*}=\sum_{\alpha=0}^{r}(-1)^{\alpha}S_{\alpha}^{*}A^{*r-\alpha}_{E},
\end{equation}
or equivalently by the inductive formula
\begin{equation}\label{N28}
T_{0}^{*}=\mathbb{I},\quad T_{r}^{*}=(-1)^{r}S_r^{*}\mathbb{I}+A_{E}^{*}\circ T_{r-1}^{*}, \quad 1\leq r\leq n,
\end{equation}
Notice that, by Cayley-Hamiliton theorem, we have $T_{n+1}^{*}=0$. Moreover, $T_{r}^{*}$ are also self-adjoint and commutes with $A_{E}^{*}$.

It is important to note that the operators $T^{*}_{r}$ depend on the choice of the transversal bundle $\mathrm{tr}(TM)$ and the screen distribution $S(TM)$.  Suppose a screen distribution $S(T M)$ changes to another screen $S(T M)'$. The following are some of the local transformation equations due to this change (see \cite[p. 87]{db} for more details):
\begin{align}\label{KKprime}
 W'_{i} & = \sum_{j = 1}^{n}  W_{i}^{j} \left(W_{j} - \epsilon_{j} c_{j} E\right), \\\label{NNprime}
N'(X) & = N -\frac{1}{2}g(W, W) E +  W, \\  \label{unique}
 A_{E}^{'*}X & = A_{E}^{*}X+B(X,N-N')E,\\
\nabla'_{X}Y & = \nabla_{X}Y + B(X, Y)\{\frac{1}{2}g(W, W)E - W\},
\end{align}
for any $X$, $Y\in\Gamma( T M|_{\mathcal{U}})$, where $W = \sum_{i=1}^{n} c_{i}W_{i}$, $\{W_{i}\}$ and $\{W_{i}'\}$ are the local orthonormal bases of $S(T M)$ and $S(T M)'$ with respective transversal sections $N$ and $N'$ for the same null section $B$. Here $c_{i}$ and $W_{i}^{j}$ are smooth functions on $\mathcal{U}$ and $\{\epsilon_{1}, \cdots, \epsilon_{n}\}$ is the signature of the basis $\{W_{1}, \cdots, W_{n}\}$.  Denote by $\omega$ is the dual $1$-form of $W$, characteristic vector field of the screen change, with respect to the induced metric $g=\overline{g}|_{M}$, that is, 
\begin{equation}\label{charaForm}
 \omega(X) = g(X, W),\quad \forall\; X\in\Gamma(T M).
\end{equation}  
 Consider an orthogonal basis $\{Z_{i}\}$, for $i\in\{1,\cdots,n\}$, which diagonalizes $A_{E}^{'*}$ and $A_{E}^{*}$. Let $k_{i}'$ and $k_{i}$ be the eigenvalues corresponding to eigenvector $Z_{i}$. Then, from (\ref{unique}) we have  $(k_{i}'-k_{i})Z_{i}= - B(Z_{i},W)E$, which shows that the eigenvalues changes under the change of the screen distribution. Since the generalized expansion $\Theta_{r}$ depends on on the eigenvalues $k_{i}$, i.e. $\Theta_{r}=(-1)^{r}S_{r}^{*}=(-1)^{r}\sigma_{r}(k_{1},\cdots,k_{n})$, then a change of $N$ will cause a change in it. Now, let $\{\Theta, T^{*}_{r}\}$ and $\{\Theta', T^{*'}_{r}\}$ be two sets of the above objects under a change in $N$. Applying recurrence relation (\ref{N28}) and the fact that $T_{r}Z_{i}=(-1)^{r}S_{r}^{*i}Z_{i}$, we have 
\begin{align}\label{m1}
 T^{*'}_{r}Z_{i} & =\Theta_{r}'\mathbb{I}+(-1)^{r-1}S_{r-1}^{*i'}A_{E}^{*'}Z_{i},\\
 T^{*}_{r}Z_{i} & =\Theta_{r}\mathbb{I}+(-1)^{r-1}S_{r-1}^{*i}A_{E}^{*}Z_{i}.
\end{align}
Subtracting the second relation in (\ref{m1}) from the first and using relation (\ref{unique}) with $X=Z_{i}$, we deduce that the operators $T^{*}_{r}$ and $T^{*'}_{r}$ are related by the following equation.
\begin{align}\label{Newton}
 T^{*'}_{r}=T_{r}+(\Theta_{r}'-\Theta_{r})\mathbb{I}+\theta_{r}A_{E}^{*}+B(T^{*'}_{r-1}, N-N')E,
\end{align}
where $\theta_{r}:=(-1)^{r-1}(S_{r-1}^{*i'}-S_{r-1}^{*i})$. It is easy to see that the tensor $T^{*}_{r}$ is unique if and only if the null hypersurface $M$ is totally geodesic. For more details on Newton transformations and their properties, we refer the reader to \cite{krz}, \cite{woj} and many more references therein.

Let  $S^{*\beta}_{r} = \sigma_{r}(\kappa_{0}^{*}, \kappa_{1}^{*},\cdots,\kappa_{\beta-1}^{*}, \kappa_{\beta+1}^{*},\cdots \kappa_{n}^{*})$ for $1\leq \beta\leq n$.  A half-lightlike submanifold $M$ of an $(n+3)$-dimensional almost contact metric manifold $\overline{M}$ is called \textit{$r$-umbilical} (resp. \textit{$r$-maximal}) \cite{ati} if for all $i,j\in\{1,\cdots,n\}$, we have
\begin{equation}\label{N29}
 S^{*i}_{r} = S^{*j}_{r},\quad \mbox{(resp. $H_{r}^{*}=0$)},\quad 1\leq r\leq n.
\end{equation} 
 Then, the following algebraic properties of $T_{r}^{*}$ are well-known (see \cite{ati},   \cite{krz}, \cite{woj} and references therein for details).
\begin{align}
 T_{r}^{*}Z_{\beta} & =(-1)^{r}S^{*\beta}_{r}Z_{\beta},\label{N30}\\
 \mathrm{tr}(T_{r}^{*}) & =(-1)^{r}(n + 1-r) S_{r}^{*},\label{N31}\\
 \mathrm{tr}(A_{E}^{*}\circ T_{r-1}^{*}) & =(-1)^{r-1}rS_{r}^{*},\label{N32}\\
 \mathrm{tr}(A_{E}^{*2}\circ T_{r-1}^{*}) & =(-1)^{r}(-S_{1}^{*} S_{r}^{*} + (r + 1) S_{r+1}^{*}),\label{N33}\\
 \mathrm{tr}(T_{r-1}^{*}\circ\nabla_{X} A_{E}^{*})  & = (-1)^{r-1}X(S_{r}^{*}),\quad \forall\,X\in\Gamma(TM).\label{N34}
\end{align}
Next, we define the divergence of the operator $T_{r}^{*}:\Gamma(TM)\rightarrow \Gamma(TM)$ as the vector field $\mathrm{div}^{\nabla}(T_{r}^{*})\in\Gamma(TM)$ and given by
\begin{equation}\label{N35}
 \mathrm{div}^{\nabla}(T_{r}^{*})=\mathrm{tr}(\nabla T_{r}^{*})=\sum_{\beta=0}^{n}(\nabla_{Z_{\beta}}T_{r}^{*})Z_{\beta}.
\end{equation}
In line with (\ref{N29}), we can see that the SAC half-lightlike submanifold given in Example \ref{SAC} is $r$-minimal with $1\le r\le 6$.

\section{Fundamental SAC equations of $A_{N}$}\label{dist}

In this section, we derive SAC equations of $A_{N}$ from those of $A_{E}^{*}$. We use some of them to generalize some known results of \cite{dusa1}.
Let $TM=\mathrm{span}\{Z_{0}=E,Z_{1},\cdots,Z_{n}\}$ and $S(TM)=\mathrm{span}\{Z_{1},\cdots,Z_{n}\}$. 
\begin{proposition}\label{LEM1}
 Let $M$ be a SAC half-lightlike submanifold of an almost contact manifold $\overline{M}$ admitting a semi-symmetric non-metric connection. Let $S_{1}$ and $S_{1}^{*}$ be the $1^{st}$ order mean curvatures corresponding to the two shape operators $A_{N}$ and $A_{E}^{*}$ respectively. Then,
 \begin{equation}\label{black}
  S_{1}=\varphi S_{1}^{*}-an.
 \end{equation}
\end{proposition}
\begin{proof}
 From $H_{1}=\frac{1}{n+1} \mathrm{tr}(A_{N})$ and $H_{1}=\dbinom{n+1}{1}^{-1}S_{1}$, we have $S_{1}= \mathrm{tr}(A_{N})$. From the last equation and the fact that $M$ is SAC half-lightlike submanifold, we have $ S_{1}=\mathrm{tr}(A_{N})|_{S(TM)}=\varphi\mathrm{tr}(A_{E}^{*})-an$ which completes the proof.
\end{proof}

In the next proposition, we generalize  Proposition \ref{LEM1}.
\begin{proposition}\label{PROP1}
 Let $M$ be a SAC half-lightlike submanifold of an almost contact manifold $\overline{M}$ admitting a semi-symmetric non-metric connection. Let $S_{r}$ and $S_{r}^{*}$ be the $r$-th mean curvatures corresponding to the two shape operators $A_{N}$ and $A_{E}^{*}$ respectively. Then, for all $r\ge1$ we have
 \begin{equation}\label{N61}
 S_{r}=\varphi^{r}S_{r}^{*}+J_{r}^{*}(a,\varphi),
 \end{equation}
 where for a given $A_{E}^{*}$, $J_{r}^{*}$  are smooth functions  in $a$  and $\varphi$ given by
\begin{align}\label{N73}
 J_{r}^{*}(a,\varphi) &=\sum_{0<i_{1}<\cdots < i_{r}\leq n}(-1)^{r} a^{r}\nonumber\\
 & +\sum_{0\leq i_{1}< \cdots< i_{r}\leq n} \sum_{j=1}^{r-1}(-1)^{r+j}e_{j}(\kappa_{i1}^{*},\cdots,\kappa_{ir}^{*}) a^{r-j}\varphi^{j}.
\end{align}
\end{proposition}
 \begin{proof}
 Let $\kappa_{0}^{*},\cdots,\kappa_{n}^{*}$ be the eigenvalues (principal curvatures) of  $A_{E}^{*}$ and 
 consider a linear factorization of a $k^{th}$-degree monic polynomial in $t$ below;
 \begin{equation}\label{go}
  \prod\limits_{i=1}^{k} (t-X_{i})=\sum_{s=0}^{k}(-1)^{s}e_{s}(X_{1},\cdots,X_{k})t^{k-s},   
 \end{equation}
where $e_{s}$ denotes the $s^{th}$-degree symmetric function in variables $X_{1},\cdots,X_{k}$.  Thus, if $S_{r}$ is the $r^{th}$ mean curvatures of $A_{N}$, then we have from the definition of $S_{r}$ and (\ref{go}) that
\begin{align*}
 S_{r}&=\sum_{0\leq i_{1}< \cdots< i_{r}\leq n} \kappa_{i_{1}}\cdots \kappa_{i_{r}}=\sum_{0\leq i_{1} <\cdots < i_{r}\leq n}  \prod\limits_{j=1}^{r} (\varphi \kappa_{i_{j}}^{*}-a)\\ 
 &=\varphi^{r}\sum_{0\leq i_{1}<\cdots < i_{r}\leq n}e_{r}(\kappa_{i_{1}}^{*},\cdots,\kappa_{i_{r}}^{*})+\sum_{0<i_{1}<\cdots < i_{r}\leq n}(-1)^{r} a^{r}\\
      &+\sum_{0\leq i_{1}<\cdots < i_{r}\leq n}(-1)^{r}\{-e_{1}(\kappa_{i_1}^{*},\cdots,\kappa_{ir}^{*})a^{r-1}\varphi+\cdots\nonumber\\
      &+(-1)^{r-1}e_{r-1}(\kappa_{i_1}^{*},\cdots,\kappa_{i_r}^{*})a\varphi^{r-1}\}\\
      &=\varphi^{r}S_{r}^{*} +\sum_{0<i_{1}<\cdots < i_{r}\leq n}(-1)^{r} a^{r}\nonumber\\
      &+\sum_{0\leq i_{1}< \cdots< i_{r}\leq n} \sum_{j=1}^{r-1}(-1)^{r+j}e_{j}(\kappa_{i_1}^{*},\cdots,\kappa_{i_r}^{*}) a^{r-j}\varphi^{j},
\end{align*}
which  proves (\ref{N61}) and (\ref{N73}), hence the proof.
\end{proof} 
Notice that $J_{r}^{*}(0,\varphi)=0$, which is the case when the structure vector field $\xi$ belongs to $S(TM)$. From (\ref{N61}), we have $S_{1}=\varphi S_{1}^{*}+J_{1}^{*}(a,\varphi)$ and from (\ref{N73}), we can see that $J_{1}^{*}(a,\varphi)=-an$. Thus, $S_{1}=\varphi S_{1}^{*}-an$, which is Proposition \ref{LEM1}.

In what follows, we construct a SAC half-lightlike submanifold in which $\xi\in\Gamma(S(TM))$ of an indefinite almost contact manifold $\overline{M}=\mathbb{R}_{2}^{9}$, admitting a semi-symmetric non-metric connection. Notice from (\ref{N26}) that under the above conditions, we have $a=0$ and therefore the corresponding SAC half-lightlike submanifold is actually screen conformal. 
\begin{example}
{\rm
Let $\overline{M}=(\mathbb{R}^{9}_{2},\overline{\phi}_{0}, \xi, \eta, \overline{g})$ be an almost contact manifold,  where $\overline{g}$ is of signature $(-,-,+,+,+,+,+,+,+)$ with respect to the canonical basis
$$
 (\partial x_{1},\partial x_{2},\partial x_{3},\partial x_{4},\ \partial x_{5},\partial x_{6},\partial x_{7},\partial x_{8},\partial z),
$$
where $(x_{1},\cdots,x_{8},z)$ are the usual coordinates on $\overline{M}$.  Let  $\overline{\phi}_{0}\partial x_{1}=\partial x_{2}$, $\overline{\phi}_{0}\partial x_{2}=-\partial x_{1}$, $\overline{\phi}_{0}\partial x_{3}=\partial x_{4}$, $\overline{\phi}_{0}\partial x_{3}=-\partial x_{4}$, $\overline{\phi}_{0}\partial x_{5}=\partial x_{6}$, $\overline{\phi}_{0}\partial x_{6}=-\partial x_{5}$, $\overline{\phi}_{0}\partial x_{7}=\partial x_{8}$, $\overline{\phi}_{0}\partial x_{8}=-\partial x_{7}$ and $\overline{\phi}_{0}\partial z=0$. consider a submanifold of $\overline{M}$ defined by 
\begin{equation*}
 M=\{(x_{1},\cdots,x_{8},z)\in \mathbb{R}_{2}^{9}: x_{1}=x_{2}-x_{7}-x_{8}\}.
\end{equation*}
Following simple calculations, we can see that the vectors $E=\partial x_{1}-\partial x_{2}+\partial x_{7}+\partial x_{8}$, is lightlike with corresponding lightlike transversal vector $N=-\frac{1}{4}(\partial x_{1}-\partial x_{2}-\partial x_{7}-\partial x_{8})$. Hence, $\mathrm{Rad}\,TM=\mathrm{span}\{E\}$ and $l\mathrm{tr}(TM)=\mathrm{span}\{N\}$. Furthermore, the vector fields  $Z_{1}=e^{x_{1}}\partial x_{3}$, $Z_{2}=e^{x_{1}}\partial x_{4}$, $Z_{3}=e^{x_{1}}\partial x_{6}$, $Z_{4}= \partial x_{1}+\partial x_{2}-\partial x_{7}+\partial x_{8}$, $Z_{5}=-\frac{1}{4}(\partial x_{1}+\partial x_{2}+\partial x_{7}-\partial x_{8})$ and $Z_{6}=\partial z=\xi$,
spans $S(TM)$. Also, $S(TM^{}=\mathrm{span}\{W\}$, where $W=-e^{x_{1}}\partial x_{5}$. Hence, $M$ is a half-lightlike submanifold of $\overline{M}$. By straightforward calculations, we have $[E,N]=0$, $[Z_{j},N]=0$, for all $j=4,5,6$ and $[Z_{i},N]=\frac{1}{4}Z_{i}$, for $i=1,2,3$. In a similar way, we have $[Z_{j},E]=0$ for all $j=4,5,6$ and $[Z_{i},E]=-Z_{i}$, for all $i=1,2,3$. Further, $[Z_{1},Z_{4}]=-Z_{1}$, $[Z_{1},Z_{5}]=\frac{1}{4}Z_{1}$, $[Z_{2},Z_{4}]=-Z_{2}$, $[Z_{2},Z_{5}]=\frac{1}{4}Z_{2}$, $[Z_{3},Z_{4}]=-Z_{3}$, $[Z_{3},Z_{5}]=\frac{1}{4}Z_{3}$.  Additionally, $[W,E]=-W$ and $[W,N]=\frac{1}{4}W$. All other brackets vanish. Notice that $S(TM)$ is integrable. Using the fact that $\overline{\nabla}$ is a semi-symmetric non-metric connection, we get $\tau=\rho=0$, $\overline{\nabla}_{E}N=0$, $\overline{\nabla}_{Z_{1}}N=\frac{1}{4}Z_{1}$, $\overline{\nabla}_{Z_{2}}N=\frac{1}{4}Z_{2}$, $\overline{\nabla}_{Z_{3}}N=\frac{1}{4}Z_{3}$, $\overline{\nabla}_{Z_{j}}N=0$, for all $j=4,5,6$. Also, $\phi=0$, $\nabla_{E}E=0$, $\nabla_{Z_{1}}E=-Z_{1}$, $\nabla_{Z_{2}}E=-Z_{2}$, $\nabla_{Z_{3}}E=-Z_{3}$, $\nabla_{Z_{j}}E=0$, for all $j=4,5,6$. From these calculations, we deduce that $A_{N}X=-\frac{1}{4}A_{E}^{*}X$, for all $X\in\Gamma(TM)$. Thus $M$ is a SAC (in particular, screen homothetic) half-lightlike submanifold, with $\varphi=-\frac{1}{4}$ and $a=0$, of $\overline{M}$ admitting a semi-symmetric non-metric connection. Further still, if $k_{0}^{*},\cdots,k_{6}^{*}$ and $k_{0},\cdots,k_{6}$ are the principal curvatures of $A_{E}^{*}$ and $A_{N}$ with respect to the basis of eigenvectors $\{E,Z_{1},\cdots,Z_{6}\}$,  respectively, then from the above information we have $k_{0}^{*}=0$, $k_{1}^{*}=1$, $k_{2}^{*}=1$, $k_{3}^{*}=1$, $k_{4}^{*}=0$, $k_{5}^{*}=0$, $k_{6}^{*}=0$ and $k_{0}=0$, $k_{1}=-\frac{1}{4}$, $k_{2}=-\frac{1}{4}$, $k_{3}=-\frac{1}{4}$, $k_{4}=0$, $k_{5}=0$, $k_{6}=0$. Hence, the matrices of $A_{E}^{*}$ and $A_{N}$ are respectively the diagonal matrices given by: $\mathrm{diag}(0,1,1,1,0,0,0)$ and  $\mathrm{diag}(0,-\frac{1}{4},-\frac{1}{4},-\frac{1}{4},0,0,0)$. Furthermore, $$S_{r}^{*}=\sigma_{r}(k^{*}_{0},\cdots, k^{*}_{6})=\sigma_{r}(0,1,1,1,0,0,0), \;\;1\le r\le 6.$$
Notice that $S_{0}^{*}=1$, $S_{1}^{*}=3$, $S_{2}^{*}=3$, etc. In a similar way, if $S_{r}$ is the $r^{th}$ mean curvature with respect to $A_{N}$, then $$S_{r}=\sigma_{r}(k_{0},\cdots, k_{6})=\sigma_{r}(0,-\frac{1}{4},-\frac{1}{4},-\frac{1}{4},0,0,0)=\left(-\frac{1}{4}\right)^{r}S_{r}^{*}, \;\;1\le r\le 6.$$
Finally, we notice that the mean curvatures $S_{r}^{*}$ and $S_{r}$ are also conformally related, i.e., $S_{r}=\varphi^{r} S_{r}^{*}$ and $J_{r}^{*}(0,-\frac{1}{4})=0$.
}
 
\end{example}

\begin{theorem}\label{PROP3}
 Let $M$ be a SAC half-lightlike submanifold of an almost contact manifold $\overline{M}$ admitting a semi-symmetric non-metric connection. Let $T_{r}$ and $T_{r}^{*}$ be the $r$-th Newton transformations corresponding to the two shape operators $A_{N}$ and $A_{E}^{*}$ respectively. Then, for all $r\ge1$ we have
 \begin{equation}\label{N77}
 T_{r}=\varphi^{r}T_{r}^{*}+\mathcal{N}_{r}^{*}(a,\varphi),
 \end{equation}
 where $\mathcal{N}_{r}^{*}$ are operators depending on  $a$, $\varphi$ and $A_{E}^{*}$ given by 
 \begin{align}\label{N82} 
 \mathcal{N}_{r}^{*}(a,\varphi)&=\sum_{\alpha=1}^{r}(-1)^{\alpha}\{J_{\alpha}^{*}(a,\varphi)\left(\varphi A_{E}^{*}-a\mathbb{I}\right)^{r-\alpha}\nonumber\\
      &+ \varphi^{\alpha}S_{\alpha}^{*} \sum_{k=1}^{r-\alpha}(-1)^{k}\dbinom{r-\alpha}{k}\left(\varphi A_{E}^{*}\right)^{r-\alpha-k}(a\mathbb{I})^{k}\}\nonumber\\
      &+\sum_{j=1}^{r}(-1)^{j}\dbinom{r}{j}(\varphi A_{E}^{*})^{r-j}(a\mathbb{I})^{j}. 
\end{align}
\end{theorem}

\begin{proof}
 Using the fact that $M$ is SAC half-lightlike submanifold, the definition of $T_{r}$,  (\ref{N78}) and Proposition \ref{PROP1} we get
 \begin{align*}
 T_{r}&=\sum_{\alpha=0}^{r}(-1)^{\alpha}S_{\alpha}A^{r-\alpha}_{N} =A_{N}^{r}+\sum_{\alpha=1}^{r}(-1)^{\alpha}S_{\alpha}A^{r-\alpha}_{N}\nonumber\\
      &=\left(\varphi A_{E}^{*}-a\mathbb{I}\right)^{r}+\sum_{\alpha=1}^{r}(-1)^{\alpha}\left(\varphi^{\alpha}S_{\alpha}^{*}+J_{\alpha}^{*}(a,\varphi)\right)\left(\varphi A_{E}^{*}-a\mathbb{I}\right)^{r-\alpha}.
\end{align*}
Applying the binomial theorem, the above equation leads to
\begin{align}\label{N80}
 T_{r}&=\varphi^{r}A_{E}^{*r}+\sum_{j=1}^{r}(-1)^{j}\dbinom{r}{j}(\varphi A_{E}^{*})^{r-j}(a\mathbb{I})^{j} +\sum_{\alpha=0}^{r}(-1)^{\alpha}\{\varphi^{\alpha}S_{\alpha}^{*}+J_{\alpha}^{*}(a,\varphi)\}\nonumber\\
 & \times \{\left(\varphi A_{E}^{*}\right)^{r-\alpha} +\sum_{k=1}^{r-\alpha}(-1)^{k}\dbinom{r-\alpha}{k}\left(\varphi A_{E}^{*}\right)^{r-\alpha-k}(a\mathbb{I})^{k}  \}.  
\end{align}
Expanding the two brackets in (\ref{N80}) gives
\begin{align} 
   &T_{r} =\varphi^{r}A_{E}^{*r}+\varphi^{r}\sum_{\alpha=1}^{r}(-1)^{\alpha} S_{\alpha}^{*}A_{E}^{*r-\alpha} +\sum_{\alpha=1}^{r}(-1)^{\alpha}\{J_{\alpha}^{*}(a,\varphi)\left(\varphi A_{E}^{*}-a\mathbb{I}\right)^{r-\alpha}\nonumber\\
      &+ \varphi^{\alpha}S_{\alpha}^{*} \sum_{k=1}^{r-\alpha}(-1)^{k}\dbinom{r-\alpha}{k}\left(\varphi A_{E}^{*}\right)^{r-\alpha-k}(a\mathbb{I})^{k}\}  +\sum_{j=1}^{r}(-1)^{j}\dbinom{r}{j}(\varphi A_{E}^{*})^{r-j}(a\mathbb{I})^{j},\nonumber
\end{align} 
which gives
$
 T_{r}=\varphi^{r}T_{r}^{*}+\mathcal{N}_{r}^{*}(a,\varphi),
$
which completes the proof.
\end{proof}

From now on, we shall write $J_{r}^{*}$ instead of $J_{r}^{*}(a,\varphi)$ and $\mathcal{N}_{r}^{*}$ instead of $\mathcal{N}_{r}^{*}(a,\varphi)$.

Next, we use Proposition \ref{PROP1} above to state the following;
\begin{proposition}\label{PROP2}
 Let $M$ be a SAC half-lightlike submanifold of an almost contact metric manifold $\overline{M}$ admitting a semi-symmetric non-metric connection. Let $S_{r}^{*}$ and $T_{r}^{*}$ denote the $r$-th mean curvature and Newton transformations with respect to $A_{E}^{*}$ respectively. Then, for all $r\ge1$ we have
 \begin{align}
 \mathrm{tr}(T_{r})&=\varphi^{r}\mathrm{tr}(T_{r}^{*})+(-1)^{r}(n + 1-r) J_{r}^{*};\label{N62}\\
 \mathrm{tr}(A_{N}\circ T_{r-1})&=\varphi^{r}\mathrm{tr}(A_{E}^{*}\circ T_{r-1}^{*})+(-1)^{r-1}rJ_{r}^{*};\label{N63}\\
 \mathrm{tr}(A_{N}^{2}\circ T_{r-1})&=\varphi^{r+1}\mathrm{tr}(A^{*2}_{E}\circ T_{r-1}^{*})+(-1)^{r}\{\varphi S_{1}^{*}J_{r}^{*}\nonumber\\
 &-an\varphi^{r}S_{r}^{*}-an J_{r}^{*}+(r+1)J_{r+1}^{*}\},\label{N65}
\end{align}
for any $X\in\Gamma(TM)$.
\end{proposition}
\begin{proof}
The proof follows by straightforward calculations.
\end{proof}
 Further, using (\ref{N62}) and (\ref{N77}) we deduce the following 
 \begin{corollary}
 Let $M$ be a SAC half-lightlike submanifold of an almost contact metric manifold $\overline{M}$ admitting a semi-symmetric non-metric connection. Let $S_{r}^{*}$ and $T_{r}^{*}$ denote the $r$-th mean curvature and Newton transformations with respect to $A_{E}^{*}$ respectively. Then, for all $r\ge1$, trace of $\mathcal{N}_{r}^{*}$ satisfies
  \begin{equation}\label{N83}
 \mathrm{tr}(\mathcal{N}_{r}^{*})=(-1)^{r}(n + 1-r) J_{r}^{*}.
\end{equation}
 \end{corollary}

Next, we use Theorem \ref{PROP3} to state the following 
\begin{theorem}
 Let $M$ be a SAC half-lightlike submanifold of an almost contact manifold $\overline{M}$ admitting a semi-symmetric non-metric connection. Let $T_{r}$ and $T_{r}^{*}$ be the $r$-th Newton transformations corresponding to the two shape operators $A_{N}$ and $A_{E}^{*}$ respectively. Then, the operator $\mathcal{N}_{r}^{*}$ satisfies the recurrence relation 

 \begin{align}\label{N85}
 &\mathcal{N}_{1}^{*}=an\mathbb{I},\nonumber\\
  &\mathcal{N}_{r}^{*}=(-1)^{r}J_{r}^{*}\mathbb{I}-a\varphi^{r-1}T_{r-1}^{*}-a\mathcal{N}_{r-1}^{*}+\varphi A_{E}^{*}\circ \mathcal{N}_{r-1}^{*},\;\;r\ge2.
 \end{align}
\end{theorem}
\begin{proof}
 Using the inductive formula (\ref{N28}), (\ref{N61}), (\ref{N77}) and equation (\ref{N26}) for a SAC half-lightlike submanifold, we obtain the desired equation (\ref{N85}).
\end{proof}
Notice that (\ref{N85}) implies that
\begin{corollary}
Let $M$ be a SAC half-lightlike submanifold of an almost contact manifold $\overline{M}$ admitting a semi-symmetric non-metric connection. Let $S_{r}$ and $S_{r}^{*}$ be the $r$-th mean curvatures corresponding to the two shape operators $A_{N}$ and $A_{E}^{*}$ respectively. Then, for all $r\ge2$ we have 
\begin{align*}
  \mathrm{tr}(A_{E}^{*}\circ\mathcal{N}_{r-1}^{*})&=(-1)^{r-1}\varphi^{-1}\{rJ_{r}^{*}+a(n+2-r)(\varphi^{r-1}S_{r-1}^{*}+J_{r-1}^{*})\},\\
  \mathrm{tr}(A^{*2}_{E}\circ\mathcal{N}_{r-1}^{*})&=\varphi^{-1}\mathrm{tr}(A^{*}_{E}\circ\mathcal{N}_{r}^{*})-a\varphi^{-1}\mathrm{tr}(A^{*}_{E}\circ\mathcal{N}_{r-1}^{*})\\
  &+a\varphi^{r-2}\mathrm{tr}(A^{*}_{E}\circ T_{r-1}^{*})+\varphi^{-1}(-1)^{r-1}J_{r}^{*}S_{1}^{*}.
\end{align*}
\end{corollary}

\begin{proposition}\label{TN1}
Let $(M,g,S(TM))$ be an $(n+1)$-dimensional SAC half-lightlike submanifold of an indefinite nearly cosymplectic manifold $\overline{M}$ admitting a semi-symmetric non-metric connection.  Denote by $\nabla$ the induced connection on $TM$. Then
\begin{align}\label{N36}
 & g (\mathrm{div}^{\nabla}(T_{r}),X)\nonumber\\ 
 &= (-1)^{r-1}\lambda(X)(\varphi^{r}E(S_{r}^{*})+E(\varphi^{r})S_{r}^{*}+E(J_{r}^{*}))+g((\nabla_{E}A_{N})T_{r-1}E,X)\nonumber\\
 &+\varphi g(\mathrm{div}^{\nabla}(\varphi^{r-1}T_{r-1}^{*}),A_{E}^{*}X)+\varphi g(\mathrm{div}^{\nabla}(\mathcal{N}_{r-1}^{*}),A_{E}^{*}X)\nonumber\\
 &-a g(\mathrm{div}^{\nabla}(\varphi^{r-1}T_{r-1}^{*}),X)-a g(\mathrm{div}^{\nabla}(\mathcal{N}_{r-1}^{*}),X)\nonumber\\
 &+a \lambda(X)\varphi^{r-1}\mathrm{tr}(A_{E}^{*}\circ T_{r-1}^{*})+a \lambda(X)\mathrm{tr}(A_{E}^{*}\circ \mathcal{N}_{r-1}^{*})\nonumber\\
 &-\varphi^{r}\lambda(X)\mathrm{tr}(A_{E}^{*^{2}}\circ T_{r-1}^{*})-\varphi\lambda(X)\mathrm{tr}(A_{E}^{*^{2}}\circ \mathcal{N}_{r-1}^{*})\nonumber\\
 &+a\tau(X)\varphi^{r-1}\mathrm{tr}(T_{r-1}^{*})+(a-\varphi)\tau(X)\mathrm{tr}(\mathcal{N}_{r-1}^{*})\nonumber\\
 &-\varphi^{r}\tau(X)\mathrm{tr}(A_{E}^{*}\circ T_{r-1}^{*})+(a\eta(X)-\varphi\eta(A_{E}^{*}X))\mathrm{tr}(\mathcal{N}^{*}_{r-1})\nonumber\\
 &+(a\varphi^{r-1}\eta(X)-\varphi^{r}\eta(A_{E}^{*}X))\mathrm{tr}(T^{*}_{r-1})+\varphi^{r}\eta(X)\mathrm{tr}(A_{E}^{*}\circ T_{r-1}^{*})\nonumber\\
  &+(-1)^{r-1}\eta(X)r J_{r}^{*}+(\varphi^{r}\lambda(A_{E}^{*}X)-a\varphi^{r-1}\lambda(X))\mathrm{tr}(A_{E}^{*}\circ T_{r-1}^{*})\nonumber\\
 &+(\varphi\lambda(A_{E}^{*}X)-a\lambda(X))\mathrm{tr}(\mathcal{N}_{r-1}^{*})+\sum_{i=0}^{n}\big\{ \overline{g}(R(Z_{i},X)\varphi^{r-1}T_{r-1}^{*}Z_{i},N)\nonumber\\
 &+\overline{g}(R(Z_{i},X)\mathcal{N}_{r-1}^{*}Z_{i},N)+\varphi^{r} \tau(Z_{i}) B(X,T_{r-1}^{*}Z_{i})\nonumber\\
 &+\varphi \tau(Z_{i}) B(X,\mathcal{N}_{r-1}^{*}Z_{i})-a\tau(Z_{i})(\varphi^{r-1}g(X,T_{r-1}^{*}Z_{i})+g(X,\mathcal{N}^{*}_{r-1}Z_{i}))\nonumber\\
 &+\lambda(T_{r-1}^{*}Z_{i})(\varphi^{r}B(Z_{i},A_{E}^{*}X)-a\varphi^{r-1}B(X,Z_{i}))\nonumber\\
 &+ \lambda(\mathcal{N}_{r-1}^{*}Z_{i})(\varphi B(Z_{i},A_{E}^{*}X)-a B(X,Z_{i}))\nonumber\\
 &+(-\varphi^{r}\lambda(A_{E}^{*}(T_{r-1}^{*}Z_{i}))-\varphi\lambda(A_{E}^{*}(\mathcal{N}_{r-1}^{*}Z_{i}))+a\varphi^{r-1}\lambda(T_{r-1}^{*}Z_{i})\nonumber\\
 &+a\lambda(\mathcal{N}_{r-1}^{*}Z_{i}))B(X,Z_{i})+(-\varphi^{r}\eta(A_{E}^{*}(T_{r-1}^{*}Z_{i}))-\varphi\eta(A_{E}^{*}(\mathcal{N}_{r-1}^{*}Z_{i}))\nonumber\\
 &+a\varphi^{r-1}\eta(T_{r-1}^{*}Z_{i})+a\lambda(\mathcal{N}_{r-1}^{*}Z_{i}))g(X,Z_{i})+(a\eta(\mathcal{N}_{r-1}^{*}Z_{i})\nonumber\\
 &+a\varphi^{r-1}\eta(T_{r-1}^{*}Z_{i}))g(X,Z_{i})-(\varphi\eta(\mathcal{N}_{r-1}^{*}Z_{i})\nonumber\\
 &+\varphi^{r}\eta(T_{r-1}^{*}Z_{i}))g(A_{E}^{*}X,Z_{i})\big\},\quad \quad\forall\,X\in\Gamma(TM).
\end{align}
\end{proposition}
\begin{proof}
 From (\ref{N28}), (\ref{N35}) and the fact that $A_{N}$ is self-adjoint, we derive 
 \begin{align}\label{N39}
  g(\mathrm{div}^{\nabla}(T_{r}),X)&=(-1)^{r}PX(S_{r})+g(\mathrm{div}^{\nabla}(T_{r-1}),A_{N}X)\nonumber\\
                                   &+g(\sum_{\alpha=0}^{n}(\nabla_{Z_{\alpha}}A_{N})T_{r-1}Z_{\alpha},X),
 \end{align}
 for all $X\in\Gamma(TM)$.
 
 Using the definition of covariant derivative we have
 \begin{align}\label{N40}
 & g((\nabla_{Z_{i}}A_{N})T_{r-1}Z_{i},X)  = g(T_{r-1}Z_{i},(\nabla_{Z_{i}}A_{N})X)+g(\nabla_{Z_{i}}A_{N}(T_{r-1}Z_{i}),X)\nonumber\\
                                        & -g(\nabla_{Z_{i}}(A_{N}X),T_{r-1}Z_{i})+g(A_{N}(\nabla_{Z_{i}}X),T_{r-1}Z_{i}) -g(A_{N}(\nabla_{Z_{i}}T_{r-1}Z_{i}),X),
 \end{align}
for all $X\in\Gamma(TM)$.
By virtue of (\ref{N17}) and the fact that $A_{N}$ is a self-adjoint operator, equation  (\ref{N40}) reduces to
\begin{align}\label{N41}
 g((\nabla_{Z_{i}}&A_{N})T_{r-1}Z_{i},X) = g(T_{r-1}Z_{i},(\nabla_{Z_{i}}A_{N})X)\nonumber\\
 &+B(Z_{i},A_{N}X)\lambda(T_{r-1}Z_{i})+B(Z_{i},T_{r-1}Z_{i}))\lambda(A_{N}X)\nonumber\\
 &-B(Z_{i},A_{N}(T_{r-1}Z_{i}))\lambda(X)-B(Z_{i},X)\lambda(A_{N}(T_{r-1}Z_{i})\nonumber\\
 &+\eta(X)g(Z_{i},A_{N}(T_{r-1}Z_{i}))-\eta(A_{N}X)g(Z_{i},T_{r-1}Z_{i})\nonumber\\
 &-\eta(T_{r-1}Z_{i})g(Z_{i},A_{N}X)+\eta(A_{N}(T_{r-1}Z_{i}))g(Z_{i},X),
\end{align}
for any $X\in\Gamma(TM)$.

Now, applying (\ref{N37}), (\ref{N23}), (\ref{N26}) and (\ref{N41}) we derive 
\begin{align}\label{N42}
 g((\nabla_{Z_{i}}&A_{N})T_{r-1}Z_{i},X) = g(T_{r-1}Z_{i},(\nabla_{X}A_{N})Z_{i})\nonumber\\
 &+\overline{g}(R(Z_{i},X)T_{r-1}Z_{i},N)+\varphi\{\tau(Z_{i})B(X,T_{r-1}Z_{i})\nonumber\\
 &-\tau(X)B(Z_{i},T_{r-1}Z_{i})\}-a\{\tau(Z_{i})g(X,T_{r-1}Z_{i})\nonumber\\
 &-\tau(X)B(Z_{i},T_{r-1}Z_{i})\}+B(Z_{i},A_{N}X)\lambda(T_{r-1}Z_{i})\nonumber\\
 &+B(Z_{i},T_{r-1}Z_{i}))\lambda(A_{N}X)-B(Z_{i},A_{N}(T_{r-1}Z_{i}))\lambda(X)\nonumber\\
 &-B(Z_{i},X)\lambda(A_{N}(T_{r-1}Z_{i})+\eta(X)g(Z_{i},A_{N}(T_{r-1}Z_{i}))\nonumber\\
 &-\eta(A_{N}X)g(Z_{i},T_{r-1}Z_{i})-\eta(T_{r-1}Z_{i})g(Z_{i},A_{N}X)\nonumber\\
 &+\eta(A_{N}(T_{r-1}Z_{i}))g(Z_{i},X),\quad \forall\,X\in\Gamma(TM).
\end{align}
Finally, substituting (\ref{N42}) in (\ref{N39}) and using Propositions \ref{PROP1}, \ref{PROP2} and Theorem \ref{PROP3} we obtain the required equation (\ref{N36}).
\end{proof}
A semi-Riemannian manifold $\overline{M}$ of constant curvature $c$ is called a \textit{semi-Riemannian space form} \cite{db, ds2} and is denoted by $\overline{M}(c)$. Then, the curvature tensor $\overline{R}$ of $\overline{M}(c)$ is given by 
\begin{equation}\label{N24}
 \overline{R}(X,Y)Z=c\{\overline{g}(Y,Z)X-\overline{g}(X,Z)Y\},\quad\forall\, X,Y,Z\in\Gamma(T\overline{M}).
\end{equation} 
Next, using Proposition \ref{TN1} we have the following.
\begin{theorem}\label{Theo1}
 Let $M(c)$ be $r$-totally umbilical SAC half-lightlike submanifold of constant curvature $c$ and with an integrable screen distribution $S(TM)$, of an indefinite contact manifold $\overline{M}^{n+3}$, admitting a semi-symmetric non-metric connection. Suppose that $M^{\prime}$  is a leaf of $M(c)$. If the structure vector field $\xi$ is tangent to $M(c)$, but not in $S(TM)$, then,
 \begin{align}
 (-1)^{r}&(\varphi^{r}E(S_{r}^{*})+E(\varphi^{r})S_{r}^{*}+E(J_{r}^{*}))\nonumber\\
  &= -\varphi^{r}\mathrm{tr}(A_{E}^{*^{2}}\circ T_{r-1}^{*})-\varphi\mathrm{tr}(A_{E}^{*2}\circ \mathcal{N}_{r-1}^{*})-\varphi^{r}\tau(E)\mathrm{tr}(A_{E}^{*}\circ T_{r-1}^{*})\nonumber\\
  &+a\mathrm{tr}(A_{E}^{*}\circ \mathcal{N}_{r-1}^{*})+A'\mathrm{tr}( T_{r-1}^{*})+B'\mathrm{tr}(\mathcal{N}_{r-1}^{*}),\nonumber
 \end{align}
 where $A'=a\tau(E)-c\varphi^{r-1}$ and $B'=(a-\varphi)\tau(E)-a-c$.
\end{theorem}
\begin{proof}
 Since $M$ is a space form, then $\mathrm{div}^{\nabla}(T_{r})\in\Gamma(TM^{\perp})$. Thus, taking $X=E$ in (\ref{N36}) and simplifying the resultant equation while considering (\ref{N18}), (\ref{N26}) and (\ref{N24}), we get 
 \begin{align*}
  &(-1)^{r}(\varphi^{r}E(S_{r}^{*})+E(\varphi^{r})S_{r}^{*}+E(J_{r}^{*}))-a\mathrm{tr}(A_{E}^{*}\circ \mathcal{N}_{r-1}^{*})\nonumber\\
  &+\varphi^{r}\mathrm{tr}(A_{E}^{*2}\circ T_{r-1}^{*})+\varphi\mathrm{tr}(A_{E}^{*2}\circ \mathcal{N}_{r-1}^{*})-a\varphi^{r}\tau(E)\mathrm{tr}( T_{r-1}^{*})\nonumber\\
  &+(\varphi-a)\tau(E)\mathrm{tr}(\mathcal{N}_{r-1}^{*})+\varphi^{r}\tau(E)\mathrm{tr}(A_{E}^{*}\circ T_{r-1}^{*})\nonumber\\
  &+a\mathrm{tr}(\mathcal{N}_{r-1}^{*})+c(\varphi^{r-1}\mathrm{tr}( T_{r-1}^{*})+\mathrm{tr}(\mathcal{N}_{r-1}^{*}))=0,\nonumber
 \end{align*}
from which our assertion follows by re-arrangement. 
\end{proof}
From the above theorem we have;
\begin{corollary}\label{cooo}
 Let $M(c)$ be $r$-totally umbilical SAC half-lightlike submanifold of constant curvature $c$ and with an integrable screen distribution $S(TM)$, of an indefinite nearly contact manifold $\overline{M}^{n+3}$, admitting a semi-symmetric non-metric connection. Suppose that $M^{\prime}$  is a leaf in $M(c)$. If the structure vector field $\xi$ is tangent to $M(c)$ and belongs to $S(TM)$, then, $M(c)$ is a Semi-Euclidean space if and only if 
\begin{align*}
 \varphi^{r}E(S_{r}^{*})&+E(\varphi^{r})S_{r}^{*}=(-1)^{r-1}\varphi^{r}(\mathrm{tr}(A_{E}^{*2}\circ T_{r-1}^{*})+\tau(E)\mathrm{tr}(A_{E}^{*}\circ T_{r-1}^{*})).
\end{align*} 
\end{corollary}
\begin{proof}
 The proof follows easily from Theorem \ref{Theo1} using  the fact $a=0$ when $\xi\in\Gamma(S(TM))$.
\end{proof}
From Corollary \ref{cooo} we have 
\begin{align*}
  \varphi^{r} E(S_{r}^{*})+E(\varphi^{r})S_{r}^{*}&=\varphi^{r}(\mathrm{tr}(A_{E}^{*2}\circ T_{r-1}^{*})+\tau(E)\mathrm{tr}(A_{E}^{*}\circ T_{r-1}^{*}))\\
  &=r\varphi^{r}S_{r}^{*}\tau(E)+\varphi^{r}\sum_{i=1}^{n}k_{i}^{*2}S_{r-1}^{*^{\alpha}},
\end{align*}
which on simplifying gives
\begin{align}\label{duu}
 E(\varphi^{r}S_{r}^{*})=r\varphi^{r}S_{r}^{*}\tau(E)+\varphi^{r}\sum_{i=1}^{n}k_{i}^{*2}S_{r-1}^{*^{\alpha}}.
\end{align}

Notice that (\ref{duu}) recovers Theorem 4.5 of \cite{dusa1}, which says: \textit{for a conformal half-lightlike submanifold $M(c)$ with mean curvature $K$ and  $M^{\prime}$ is a totally umbilical leaf in $M(c)$, then, the submanifold  is a semi-Euclidean space if and only if $K$satisfies}
\begin{equation*}
 E(K)-K\tau(E)-K^{2}\varphi^{-1}=0.
\end{equation*}

\section{Special Minkowski integration formulae}\label{int}

In this section, we present a new set of integration formulas on a  special SAC half-lightlike submanifold $(M,g,S(TM),S(TM^{\perp}))$ of an indefinite nearly cosymplectic manifold $(\overline{M},\overline{g})$, called SAC $\overline{H}$-half lightlike submanifold, via the computation of $\mathrm{div}^{\nabla}(T_{r}\overline{H}X')$ and $\mathrm{div}^{\nabla}(T_{r}\overline{H}X'+T_{r}E)$, where $X'\in\Gamma(S(TM)^{\perp})$ and $E\in \Gamma(\mathrm{Rad}\,TM)$. We shall suppose that $M$ is closed and bounded (compact). An almost contact manifold $\overline{M}$ is said to be nearly cosymplectic if
\begin{equation}\label{eqz}
         (\overline{\nabla}_{\overline{X}} \overline{\phi})\overline{Y}+(\overline{\nabla}_{\overline{Y}} \overline{\phi})\overline{X}=0,
    \end{equation}
for any vector fields $\overline{X}$, $\overline{Y}$ on $\overline{M}$, where $\overline{\nabla}$ is the connection for the semi-Riemannian metric $\overline{g}$.  

Replacing $\overline{Y}$ by $\xi$ in (\ref{eqz}) we obtain

   \begin{equation}\label{v10}
    \overline{\nabla}_{\overline{X}} \xi =-\overline{H}\,\overline{X},
   \end{equation}
   where $\overline{H}$ is a (1,1) tensor given by $ \overline{H}\,\overline{X}=\overline{\phi}(\overline{\nabla}_{\xi} \overline{\phi})\overline{X}$. 
 The linear operator $\overline{H}$ has the properties \cite{ms}:
 \begin{align}
  & \overline{H}\,\overline{\phi} + \overline{\phi}\,\overline{H}=0,\;\;\overline{H}\xi=0,\;\;\eta\circ \overline{H}=0,\;\;(\overline{\nabla}_{X}\overline{\phi})\xi=\overline{\phi}\,\overline{H}X, \nonumber\\
  \mbox{and}\;\;& \overline{g}(\overline{H}\,\overline{X}, \overline{Y})=-\overline{g}(\overline{X}, \overline{H}\,\overline{Y})\;\;\;\; (\mbox{i.e.}\;\; \overline{H} \;\;\mbox{is skew-symmetric}),
 \end{align}
 for all $\overline{X},\overline{Y}\in \Gamma(T\overline{M})$.
 
\begin{proposition}\label{KK1}
 Let $(M,g,S(TM),S(TM^{\perp}))$ be any half-lightlike submanifold of an indefinite nearly cosymplectic manifold $(\overline{M},\overline{g})$. Then, 
 $$
 \overline{H}S(TM)^{\perp}\subset S(TM).
 $$
\end{proposition}
\begin{proof}
 On a lightlike submanifold we have 
 $$
 S(TM)^{\perp}=\{\mathrm{Rad}\,TM \oplus l\mathrm{tr}(TM)\}\perp S(TM^{\perp}).
 $$ 
 Thus, to show that $\overline{H}S(TM)^{\perp}\subset S(TM)$ it is enough to show that 
 \begin{align*}
 &\overline{H}\mathrm{Rad}\,TM\subset S(TM), \quad \overline{H}l\mathrm{tr}(TM)\subset S(TM),\\
 &\mbox{and} \quad \overline{H}S(TM^{\perp})\subset S(TM).
 \end{align*}
 Now, let $E\in\Gamma(\mathrm{Rad}\,TM)$, then by the anti-symmetry of $\overline{H}$ we have $\overline{g}(\overline{H}E,E)=0$. This shows that $\overline{H}E$ has no component along $l\mathrm{tr}(TM)$. Thus, we can see that $\overline{H}E\in\Gamma(TM\perp S(TM^{\perp}))$, which indicates that $\overline{H}E$ has a possibility of belonging to $\Gamma(TM)$. Considering this option and the fact that on a half-lightlike submanifold $\mathrm{rank}(\mathrm{Rad}\,TM)=1$, we can see that $\overline{H}\mathrm{Rad}\,TM$ is a distribution on $M$ of rank 1 such that $\overline{H}\mathrm{Rad}\,TM\cap\mathrm{Rad}\,TM$=\{0\}. Therefore, we can choose a particular $S(TM)$ containing $\overline{H}\mathrm{Rad}\,TM$ as one of its subbundles. Further, let $N\in\Gamma(l\mathrm{tr}(TM))$, then $\overline{g}(\overline{H}N,E)=-\overline{g}(N,\overline{H}E)=0$. Thus, $\overline{H}N$ has no component along  $l\mathrm{tr}(TM)$. Since $\overline{g}(\overline{H}N,N)=0$ we can see that $\overline{H}N$ has no component along $\mathrm{Rad}\,TM$. Hence, $\overline{H}N\in\Gamma(S(TM)\perp S(TM^{\perp}))$, from which we can also choose $\overline{H}N\in\Gamma(S(TM))$. Using similar reasoning as above,  we can see check that $\overline{g}(\overline{H}W,W)=0$, $\overline{g}(\overline{H}W,E)=0$ and $\overline{g}(\overline{H}W,N)=0$, for $W\in\Gamma(S(TM^{\perp}))$. From the last three equations, it is obvious that $\overline{H}W$ is neither in $\Gamma(\mathrm{Rad}\,TM)$ nor in $ \Gamma(l\mathrm{tr}(TM))$. Also, we infer that $\overline{H}W\notin\Gamma(S(TM^{\perp}))$. In fact, if $\overline{H}W \in\Gamma(S(TM^{\perp}))$ then there exist a non vanishing function $\omega$ such that $\overline{H}W=\omega W$. Taking the $\overline{g}$-product of the last equation with respect to $W$ we get $0=\overline{g}(\overline{H}W,W)=\omega\overline{g}(W,W)\neq0$, which is a contradiction. Thus, $\overline{H}W\in\Gamma(S(TM))$, which completes the proof. 
\end{proof}
Using Proposition \ref{KK1} we have the following;
\begin{definition}
 Let $(M,g,S(TM),S(TM^{\perp}))$ be any half-lightlike submanifold of an indefinite nearly cosymplectic manifold $(\overline{M},\overline{g})$. We say that $M$ is $\overline{H}$-half lightlike submanifold of $\overline{M}$ if 
 $$
 \overline{H}S(TM)^{\perp}\subset S(TM).
 $$
\end{definition}

\begin{proposition}\label{p1}
 Let $(M,g,S(TM),S(TM^{\perp}))$ be a SAC $\overline{H}$-half lightlike submanifold of an indefinite nearly cosymplectic manifold $(\overline{M},\overline{g})$ admitting a semi-symmetric non-metric connection. Then,
  \begin{align}
  g(\nabla_{E}T_{r}\overline{H}X',N)&=(-1)^{r}S_{r}\eta(\overline{H}X'),\label{K3}\\
  g(\nabla_{Z_{i}}T_{r}\overline{H}X',Z_{i})&=g(\overline{H}X',(\nabla_{Z_{i}}T_{r})Z_{i})+g(\nabla_{Z_{i}}^{*}T_{r}\overline{H}X',Z_{i})\nonumber\\
  &-g(\overline{H}X',(\nabla_{Z_{i}}^{*}T_{r})Z_{i})-\eta(\overline{H}X')g(Z_{i},T_{r}Z_{i})\nonumber\\
  &+(-1)^{r}S_{r}\eta(\overline{H}X'),\label{K4}
  \end{align}
for all $X'\in\Gamma(S(TM)^{\perp})$ and $\nabla^{*}$ is the metric connection on $S(TM)$.
\end{proposition}
\begin{proof}
By straight forward calculations we have
 \begin{align*}
   g(\nabla_{E}T_{r}\overline{H}X',N)=(-1)^{r}g(\nabla_{E}S_{r}\overline{H}X',N)=(-1)^{r}\overline{g}(S_{r}\nabla_{E}\overline{H}X',N).
 \end{align*}
Now, applying (\ref{N15}) to the above equation we get 
\begin{equation*}
  g(\nabla_{E}T_{r}\overline{H}X',N)=(-1)^{r}\overline{g}(S_{r}\nabla_{E}^{*}\overline{H}X',N)+(-1)^{r}S_{r}C(E,\overline{H}X').
\end{equation*}
from which we get 
\begin{equation}\label{K5}
  g(\nabla_{E}T_{r}\overline{H}X',N)=(-1)^{r}S_{r}C(E,\overline{H}X').
\end{equation}
Applying (\ref{N21}) to (\ref{K5}) with $X=E$ and $PY=\overline{H}X'$ we get 
\begin{equation*}
  \overline{g}(\nabla_{E}T_{r}\overline{H}X',N)=(-1)^{r}S_{r}\eta(\overline{H}X'),
\end{equation*}
which proves (\ref{K3}).

On the other hand, using (\ref{N17}) we derive
\begin{align}\label{K6}
 g(\nabla_{Z_{i}}T_{r}\overline{H}X',Z_{i})&=Z_{i}(g(T_{r}\overline{H}X',Z_{i}))-g(\overline{H}X',T_{r}\nabla_{Z_{i}}Z_{i})\nonumber\\
                                           &+(-1)^{r}S_{r}\eta(\overline{H}X')+\eta(Z_{i})g(\overline{H}X',T_{r}Z_{i}).
\end{align}
Now, applying the definition of covariant derivative of $T_{r}$ on (\ref{K6}) we get
\begin{align}\label{K7}
 g(\nabla_{Z_{i}}T_{r}\overline{H}X',Z_{i})&=g(\overline{H}X',(\nabla_{Z_{i}}T_{r})Z_{i})+Z_{i}(g(T_{r}\overline{H}X',Z_{i}))\nonumber\\
 &-g(\overline{H}X',\nabla_{Z_{i}}T_{r}Z_{i})+(-1)^{r}S_{r}\eta(\overline{H}X')\nonumber\\
 &+\eta(Z_{i})g(\overline{H}X',T_{r}Z_{i}).
\end{align}
Also, using (\ref{N17}) we derive
\begin{align}\label{K8}
 &Z_{i}(g(\overline{H}X',T_{r}Z_{i}))-g(\overline{H}X',\nabla_{Z_{i}}T_{r}Z_{i})\nonumber\\
 &=g(T_{r}\nabla_{Z_{i}}\overline{H}X',Z_{i})-\eta(\overline{H}X')g(Z_{i},T_{r}Z_{i})-\eta(Z_{i})g(Z_{i},T_{r}\overline{H}X').
\end{align}
Then substituting (\ref{K8}) in (\ref{K7}) and then applying (\ref{N15}) we get (\ref{K4}), which completes the proof.
\end{proof}
Since $M$ is null, the divergence $\mathrm{div}^{\nabla}(Y)$ of a vector $Y\in\Gamma(T M)$ with respect to the degenerate metric $g$ on $L$ is intrinsically defined by (see \cite[p. 136]{ds2}, for more details and references therein)
\begin{align}\label{divergence}
  \mathrm{div}^{\nabla}(Y)= \mathrm{div}^{\nabla^{*}}(Y) +  g(\nabla_{E}Y,N).
\end{align}
Now we are ready to compute $\mathrm{div}^{\nabla}(T_{r}\overline{H}X')$.
\begin{theorem}\label{The1}
Let $(M,g,S(TM),S(TM^{\perp}))$ be a SAC $\overline{H}$-half lightlike submanifold of an indefinite nearly cosymplectic manifold $(\overline{M},\overline{g})$ admitting a semi-symmetric non-metric connection. Then,
\begin{align*}
 \mathrm{div}^{\nabla}(T_{r}\overline{H}X')&=g(\mathrm{div}^{\nabla}(\varphi^{r}T_{r}^{*}),\overline{H}X')+g(\mathrm{div}^{\nabla}(\mathcal{N}_{r}^{*}),\overline{H}X')\\
 &-g(\mathrm{div}^{\nabla^{*}}(\varphi^{r}T_{r}^{*}),\overline{H}X')-g(\mathrm{div}^{\nabla^{*}}(\mathcal{N}_{r}^{*}),\overline{H}X')\\
 &+\mathrm{tr}(\nabla^{*}(\varphi^{r}T_{r}^{*}\overline{H}X'))+\mathrm{tr}(\nabla^{*}(\mathcal{N}_{r}^{*}\overline{H}X'))\\
 &-\varphi^{r}\eta(\overline{H}X')\mathrm{tr}(T_{r}^{*})-\eta(\overline{H}X')\mathrm{tr}(\mathcal{N}_{r}^{*})\\
 &+(-1)^{r}(n+1)\eta(\overline{H}X')(\varphi^{r}S_{r}^{*}+J_{r}^{*}),
 \end{align*}
 for all $X'\in\Gamma(S(TM)^{\perp})$.
\end{theorem}
\begin{proof}
 By definition of divergence (\ref{divergence}) we have
 \begin{align}\label{K9}
  \mathrm{div}^{\nabla}(T_{r}\overline{H}X')&=\mathrm{tr}(\nabla T_{r}\overline{H}X')\nonumber\\
  &=\sum_{i=1}^{n}g(\nabla_{Z_{i}}T_{r}\overline{H}X',Z_{i})+\overline{g}(\nabla_{E}T_{r}\overline{H}X',N).
 \end{align}
Now, applying (\ref{K3}) and (\ref{K4}) of Proposition \ref{p1} to (\ref{K9}) we get the desired equation which ends the proof.
\end{proof}
\begin{theorem}\label{The2}
 Let $(M,g,S(TM),S(TM^{\perp}))$ be a SAC $\overline{H}$-half lightlike submanifold of an indefinite nearly cosymplectic manifold $(\overline{M},\overline{g})$ admitting a semi-symmetric non-metric connection. Then,
 \begin{equation*}
  \mathrm{div}^{\nabla}(T_{r}E)=(-1)^{r}(A_{1}S_{r}^{*}+A_{2}J_{r}^{*}+\varphi^{r}E(S_{r}^{*})+E(J_{r}^{*})),
 \end{equation*}
 where 
$$
A_{1}=\varphi^{r}\tau(E)+E(\varphi^{r})-\varphi^{r}S_{1}^{*}\;\; \mbox{and} \;\;A_{2}=\tau(E)-S_{1}^{*},
$$
for any $E\in\Gamma(\mathrm{Rad}\,TM)$.
\end{theorem}
\begin{proof}
 By straightforward calculations we have
  \begin{equation}\label{K11}
  \mathrm{div}^{\nabla}(T_{r}E) =\mathrm{tr}(\nabla T_{r}E) =\sum_{i=1}^{n}g(\nabla_{Z_{i}}T_{r}E,Z_{i})+g(\nabla_{E}T_{r}E,N).
 \end{equation}
 But 
 \begin{align}\label{K12}
  \overline{g}(\nabla_{E}T_{r}E,N)=(-1)^{r}(E(S_{r})+S_{r}\overline{g}(\nabla_{E}E,N)).
 \end{align}
Now, Applying (\ref{N16}) and Proposition \ref{PROP1} to (\ref{K12}) we get 
\begin{align}\label{K13}
 \overline{g}(\nabla_{E}T_{r}E,N)=(-1)^{r}(E(\varphi^{r}S_{r}^{*})+E(J_{r}^{*})+\tau(E)(\varphi^{r}S_{r}^{*}+J_{r}^{*})).
\end{align}
Also, 
\begin{equation}\label{K14}
 g(\nabla_{Z_{i}}T_{r}E,Z_{i})=(-1)^{r-1}(\varphi^{r}S_{r}^{*}+J_{r}^{*})g(A_{E}^{*}Z_{i},Z_{i}).
\end{equation}
Finally, replacing (\ref{K12}) and (\ref{K14}) in (\ref{K11}), we get desired result. 
\end{proof}
From Theorems \ref{The1} and \ref{The2} we have;
\begin{theorem}\label{The3}
Let $(M,g,S(TM),S(TM^{\perp}))$ be a SAC $\overline{H}$-half lightlike submanifold of an indefinite nearly cosymplectic manifold $(\overline{M},\overline{g})$ admitting a semi-symmetric non-metric connection. Then,
\begin{align*}
 \mathrm{div}^{\nabla}(T_{r}(\overline{H}X'+E))&=g(\mathrm{div}^{\nabla}(\varphi^{r}T_{r}^{*}),\overline{H}X')+g(\mathrm{div}^{\nabla}(\mathcal{N}_{r}^{*}),\overline{H}X')\\
 &-g(\mathrm{div}^{\nabla^{*}}(\varphi^{r}T_{r}^{*}),\overline{H}X')-g(\mathrm{div}^{\nabla^{*}}(\mathcal{N}_{r}^{*}),\overline{H}X')\\
 &+\mathrm{tr}(\nabla^{*}(\varphi^{r}T_{r}^{*}\overline{H}X'))+\mathrm{tr}(\nabla^{*}(\mathcal{N}_{r}^{*}\overline{H}X'))\\
 &+(-1)^{r}(B_{1}S_{r}^{*}+B_{2}J_{r}^{*}+\varphi^{r}E(S_{r}^{*})+E(J_{r}^{*})),
 \end{align*}
  where 
 \begin{equation*}
  B_{1}=A_{1}+\varphi^{r}r\eta(\overline{H}X')\quad\mbox{and}\quad B_{2}=A_{2}+r\eta(\overline{H}X'),
 \end{equation*}
 for all $X'\in\Gamma(S(TM)^{\perp})$.
 \end{theorem}
\begin{proof}
 The proof follows directly from Theorems \ref{The1} and \ref{The2}, equation (\ref{N62}) and the fact $\mathrm{div}^{\nabla}(T_{r}(\overline{H}X'+E))=\mathrm{div}^{\nabla}(T_{r}\overline{H}X')+\mathrm{div}^{\nabla}(T_{r}E)$.
\end{proof}
Let $dV_{\overline{M}}$ be the volume element of $\overline{M}$ with respect to $\overline{g}$ and a given orientation. Then, we denote the volume form on $M$ by 
$$
dV=i_{N}dV_{\overline{M}},
$$
where $i_{N}$ is the contraction with respect to the vector field $N$. From the above theorem we have the following.
\begin{theorem}\label{The4}
 Let $(M,g,S(TM),S(TM^{\perp}))$ be a compact SAC $\overline{H}$-half lightlike submanifold of an indefinite nearly cosymplectic manifold $(\overline{M},\overline{g})$ admitting a semi-symmetric non-metric connection. Then,
 \begin{align}
  (-1)^{r-1}&\int_{M}(B_{1}S_{r}^{*}+B_{2}J_{r}^{*}+\varphi^{r}E(S_{r}^{*})+E(J_{r}^{*}))\mathrm{d}V\nonumber\\
  &=\int_{M} g(\mathrm{div}^{\nabla}(\varphi^{r}T_{r}^{*}),\overline{H}X')\mathrm{d}V+\int_{M}g(\mathrm{div}^{\nabla}(\mathcal{N}_{r}^{*}),\overline{H}X')\mathrm{d}V\nonumber\\
  &-\int_{M}g(\mathrm{div}^{\nabla^{*}}(\varphi^{r}T_{r}^{*}),\overline{H}X')\mathrm{d}V-\int_{M}g(\mathrm{div}^{\nabla^{*}}(\mathcal{N}_{r}^{*}),\overline{H}X')\mathrm{d}V\nonumber\\
  &+\int_{M}\mathrm{tr}(\nabla^{*}(\varphi^{r}T_{r}^{*}\overline{H}X'))\mathrm{d}V+\int_{M}\mathrm{tr}(\nabla^{*}(\mathcal{N}_{r}^{*}\overline{H}X'))\mathrm{d}V,\nonumber
 \end{align}
  where 
 \begin{equation*}
  B_{1}=A_{1}+\varphi^{r}r\eta(\overline{H}X')\quad\mbox{and}\quad B_{2}=A_{2}+r\eta(\overline{H}X'),
 \end{equation*}
 for all $X'\in\Gamma(S(TM)^{\perp})$.
\end{theorem}
\begin{proof}
 Since $M$ is compact, then   applying Stokes' theorem we see that 
 \begin{equation}\label{K20}
  \int_{M}\mathrm{div}^{\nabla}(T_{r}(\overline{H}X'+E))\mathrm{d}V=0, \quad \forall\,X'\in \Gamma(S(TM)^{\perp}).
 \end{equation}
Now, using (\ref{K20}) and Theorem \ref{The3} we get the desired result.
\end{proof}
In particular, if $\overline{M}$ is an indefinite cosymplectic manifold (i.e., $\overline{H}=0$), then
\begin{equation*}
 \int_{M}(B_{1}S_{r}^{*}+B_{2}J_{r}^{*}+\varphi^{r}E(S_{r}^{*})+E(J_{r}^{*}))\mathrm{d}V=0.
\end{equation*}
\begin{corollary}Theo1
 Let $(M,g,S(TM),S(TM^{\perp}))$ be a compact SAC  $\overline{H}$-half lightlike submanifold of an indefinite nearly cosymplectic manifold $(\overline{M},\overline{g})$ admitting a semi-symmetric non-metric connection. If $\xi\in\Gamma(S(TM))$, then
  \begin{align}
  (-1)^{r-1}&\int_{M}(B_{1}S_{r}^{*}+\varphi^{r}E(S_{r}^{*}))\mathrm{d}V=\int_{M} g(\mathrm{div}^{\nabla}(\varphi^{r}T_{r}^{*}),\overline{H}X')\mathrm{d}V\nonumber\\
  &-\int_{M}g(\mathrm{div}^{\nabla^{*}}(\varphi^{r}T_{r}^{*}),\overline{H}X')\mathrm{d}V+\int_{M}\mathrm{tr}(\nabla^{*}(\varphi^{r}T_{r}^{*}\overline{H}X'))\mathrm{d}V,\nonumber
 \end{align}
 for all $X'\in\Gamma(S(TM)^{\perp})$.
\end{corollary}
When $M$ is a SAC $\overline{H}$-half lightlike submanifold of constant sectional curvature, then $\mathrm{div}^{\nabla}(T_{r})\in\Gamma(TM^{\perp})$ and hence we have the following;
\begin{corollary}
  Let $(M,g,S(TM),S(TM^{\perp}))$ be a compact SAC$\overline{H}$-half lightlike submanifold of constant curvature of an indefinite nearly cosymplectic manifold $(\overline{M},\overline{g})$ admitting a semi-symmetric non-metric connection. Then,
 \begin{align}
  (-1)^{r-1}&\int_{M}(B_{1}S_{r}^{*}+B_{2}J_{r}^{*}+\varphi^{r}E(S_{r}^{*})+E(J_{r}^{*}))\mathrm{d}V\nonumber\\
  &=-\int_{M}g(\mathrm{div}^{\nabla^{*}}(\varphi^{r}T_{r}^{*}),\overline{H}X')\mathrm{d}V-\int_{M}g(\mathrm{div}^{\nabla^{*}}(\mathcal{N}_{r}^{*}),\overline{H}X')\mathrm{d}V\nonumber\\
  &+\int_{M}\mathrm{tr}(\nabla^{*}(\varphi^{r}T_{r}^{*}\overline{H}X'))\mathrm{d}V+\int_{M}\mathrm{tr}(\nabla^{*}(\mathcal{N}_{r}^{*}\overline{H}X'))\mathrm{d}V,\nonumber
 \end{align}
 for all $X'\in\Gamma(S(TM)^{\perp})$.
\end{corollary}



\begin{thebibliography}{xxx}
\bibitem{ati} C. Atindogb\'e and H. T. Fotsing, Newton transformations on null hypersurfaces,  Commun. Math. 23 (2015), no. 1, 57-83.
\bibitem{krz} K. Andrzejewski and Paweł G. Walczak,  The Newton transformation and new integral formulae for foliated manifolds, Ann. Glob. Anal. Geom. 37 (2010), no. 2, 103-111.
\bibitem{bl2}  D.E. Blair, Riemannian geometry of contact and symplectic manifolds, Progress in Mathematics, 203. Birkh\"auser Boston, Inc., Boston, MA, 2002.
\bibitem{db}  K. L. Duggal and A. Bejancu, Lightlike submanifolds of semi-Riemannian  manifolds and applications, Mathematics and Its Applications, Kluwer Academic Publishers, 1996.   
\bibitem{ds2} K. L. Duggal and B. Sahin, Differential geometry of lightlike submanifolds. Frontiers in Mathematics, Birkh\"auser Verlag, Basel, 2010.
\bibitem{dusa1} K. L. Duggal and B. Sahin, Screen conformal half-lightlike submanifolds, Int. J. Math. Math. Sci. 2004 (2004), no. 68, 3737-3753.
\bibitem{gup1} R. S. Gupta and A. Sharfuddin, Generalised Cauchy-Riemann Lightlike Submanifolds of Indefinite Kenmotsu Manifolds, Note Mat. 30 (2010) no. 2, 49-59.
\bibitem{Jin} D. H. Jin, Non-Existance of lightlike submanifolds of indefinite trans-Sasakian manifolds with non-metric $\theta$-connections, Commun. Korean Math. Soc. 30 (2015), no. 1, 35-43.
\bibitem{ma1} F. Massamba, Totally contact umbilical lightlike hypersurfaces of indefinite Sasakian manifolds, Kodai Math. J., 31 (2008), 338-358.
\bibitem{ma2}  F. Massamba, On semi-parallel lightlike hypersurfaces of indefinite Kenmotsu manifolds, J. Geom., 95 (2009), 73-89.
\bibitem{ma22} F. Massamba, On lightlike geometry in indefinite Kenmotsu manifolds. Math. Slovaca, 62 (2012), no. 2,  315-344.
\bibitem{ma3}  F. Massamba, Screen almost conformal lightlike geometry in indefinite Kenmotsu space forms, International Electronic Journal of Geometry,  5 (2012), no. 2, 36-58. 
\bibitem{ms} F. Massamba and S. Ssekajja, Some remarks on quasi generalized CR-null geometry in indefinite nearly cosymplectic manifolds, arXiv: 1604.5436v3.
\bibitem{cohen} E. Yasar, A. C. Coken, and A. Yucesan, Lightlike hypersurfaces in semi-Riemannian manifold with semi-symmetric non-metric connection, Math. Scand. 102 (2008), no. 2, 253-264.
\bibitem{woj} K. Andrzejewski, W. Kozlowski and K. Niedzialomski, Generalized Newton transformation and its applications to extrinsic geometry,  Asian J. Math. 20 (2016), no. 2, 293-322.
\end{thebibliography}
\end{document}